\documentclass[reqno]{amsart}

\usepackage{amssymb,amsthm,amsmath, mathtools, amscd, amsfonts}

\usepackage{xcolor,paralist,hyperref,fancyhdr,etoolbox}
\usepackage{graphicx}
\usepackage{cite, tikz-cd}
\usepackage{verbatim,psfrag}
\usepackage[all]{xy}
\usepackage{pinlabel}
\usepackage{anyfontsize}
\usepackage{lmodern}
\usepackage{mathabx}
\usepackage{caption}

\usepackage{etoolbox}

\makeatletter
\newcommand{\setequationspacing}{
  \setlength{\abovedisplayskip}{5pt}
  \setlength{\belowdisplayskip}{5pt}
  \setlength{\abovedisplayshortskip}{3pt}
  \setlength{\belowdisplayshortskip}{3pt}
}
\makeatother

\numberwithin{equation}{section}

\makeatletter
\def\@setfontsize#1#2#3{\fontsize{#2}{#3}\selectfont} 
\makeatother

\makeatletter
\renewenvironment{abstract}
 {\vspace{12pt}\begin{quote}\normalfont\fontsize{9pt}{11pt}\textsc{Abstract}
   \normalfont\fontsize{9pt}{11pt}\selectfont}
 {\end{quote}}
\makeatother

\newtheorem{theorem}{Theorem}[section]

\newtheorem{lemma}[theorem]{Lemma}
\newtheorem{proposition}[theorem]{Proposition}
\newtheorem{corollary}[theorem]{Corollary}

\newcommand{\blue}[1]{\textcolor{blue}{#1}}

\hypersetup{ colorlinks=true, linkcolor=blue, filecolor=black, urlcolor=black }

\begin{document}

\fontsize{11pt}{13pt}\selectfont
\title[Roots of hyperelliptic involutions]{Roots of hyperelliptic involutions and braid groups modulo their center inside mapping class groups} 
\author[Ryan Lamy]{Ryan Lamy}
\date{\today}

\address{Address}

\email{ryanlamy1@gmail.com}

\maketitle

\let\thefootnote\relax

\vspace{-5ex}

\begin{abstract}
    Let $n,k\in\mathbb{N}$ and let $S$ be the closed surface of genus $nk$. A copy of the braid group on $2k+2$ strands modulo its center is found inside $\mathrm{Mod}(S)$, provided $n\geq 3$. In particular, for $k=1$ the class of the half-twist braid inside $B_4/Z(B_4)$ is identified with a hyperelliptic involution inside $\mathrm{Mod}(S)$. As a consequence, we can show that each hyperelliptic involution inside $\mathrm{Mod}(S)$ has infinitely many square roots, and discuss their conjugacy classes. Furthermore, a copy of $\mathrm{Mod}(S_1)\cong\mathrm{SL}_2(\mathbb{Z})$ is found inside $\mathrm{Mod}(S_2)$. This subgroup contains the unique hyperelliptic involution on $S_2$. As a result, we can show that the latter admits infinitely many square and cubic roots, and discuss their conjugacy classes.
\end{abstract}

\vspace{8pt}

\section{Introduction}
\label{Section 1}

It is well known that given a closed orientable surface $S$ of genus $g$, a hyperelliptic involution on $S$ can be expressed via the composition of Dehn twists
\[
T_{c_1}\circ ... \circ T_{c_{2g+1}}\circ T_{c_{2g+1}}\circ...\circ T_{c_1},
\]
where $c_1,...,c_{2g+1}$ is a chain of simple closed curves in $S$ \cite[Section 5.1.4]{Primer}. In this paper, we explore a novel way of expressing hyperelliptic involutions on $S$ in terms of Dehn twists, one which admits infinitely many square (and sometimes cubic) roots. The construction required to assemble these hyperelliptic involutions will give rise to a copy of $\mathrm{SL}_2(\mathbb{Z})$ inside $\mathrm{Mod}(S_2)$, and copies of braid groups modulo their center inside $\mathrm{Mod}(S)$ whenever $g\geq 3$. As such, the hyperelliptic involutions constructed therein (as well as their roots) will be identified with matrices in $\mathrm{SL}_2(\mathbb{Z})$ for genus 1 and 2, and with braids for genus 3 and higher.

\vspace{9pt}

\subsection{Definitions and conventions}
\label{Section 1.2}
Before we begin, an overview of the recurring notation and definitions in this paper is given. For the following, $S$ is an orientable surface of genus $g$.
\begin{enumerate}
    \item We denote by $\mathrm{Mod}(S)$ the \textit{mapping class group} of $S$, defined as   
    \begingroup
\setequationspacing
    \[
    \mathrm{Mod}(S)\coloneq\mathrm{Homeo}^+(S)/\mathrm{Homeo}_0(S),
    \]
    \endgroup
Where $\mathrm{Homeo}^+(S)$ denotes the group of orientation preserving homeomorphisms from $S$ to itself which preserve the boundary of $S$ pointwise, and $\mathrm{Homeo}_0(S)$ denotes the subgroup consisting of those homeomorphisms which are isotopic to the identity.

\vspace{9pt}

\item Assuming the surface $S$ is closed and $g>0$, The \textit{symplectic representation} 
\begingroup
\setequationspacing
\[
\Psi:\mathrm{Mod}(S)\longrightarrow \mathrm{Sp}_{2g}(\mathbb{Z})
\]
\endgroup
is the map which describes the action of mapping classes on $H_1(S)\cong\mathbb{Z}^{2g}$ via matrices living inside the symplectic group $\mathrm{Sp}_{2g}(\mathbb{Z})$. More on this can be found in Farb/Margalit's book \cite[Chapter 6]{Primer}.

\vspace{9pt}

\item Given the above, a mapping class $\rho\in\mathrm{Mod}(S)$ is called a \textit{hyperelliptic involution} if $\rho^2=1$ and $\Psi(\rho)=-\mathrm{Id}_{2g}$.

\vspace{9pt}

\item For any simple closed curve $\alpha$ in $S$, the left Dehn twist about the curve $\alpha$ is denoted $T_{\alpha}$ and the right Dehn twist about the curve $\alpha$ is denoted $T_{\alpha}^{-1}$. We consider these as mapping classes inside $\mathrm{Mod}(S)$ rather than individual homeomorphisms.

\vspace{9pt}

\item {\spaceskip=5.2pt Suppose $G$ is a group acting on the surface $S$ by Deck transformations and} \\ $p:S\longrightarrow S/G$ is the induced branched cover. Then, $\mathrm{SMod}(S)$ denotes the subgroup of $\mathrm{Mod}(S)$ consisting of all fiber-preserving mapping classes under $p$, otherwise known as \textit{symmetric mapping classes}. A mapping class is fiber-preserving if it contains a representative homeomorphism which preserves fibers under $p$. More on this can be found in Margalit/Winarski's survey paper \cite{MR4275077}. 

\end{enumerate}

\vspace{9pt}

\subsection{Outline and results}
A concise summary of the results in this paper is now given. Let $n,k\in\mathbb{N}$ and let $S$ be the closed surface of genus $nk$. Consider the following simple closed curves in $S$.

\begin{center}  
  \captionof{figure}{}
  \vspace{7pt}
  \includegraphics[width=0.75\textwidth]{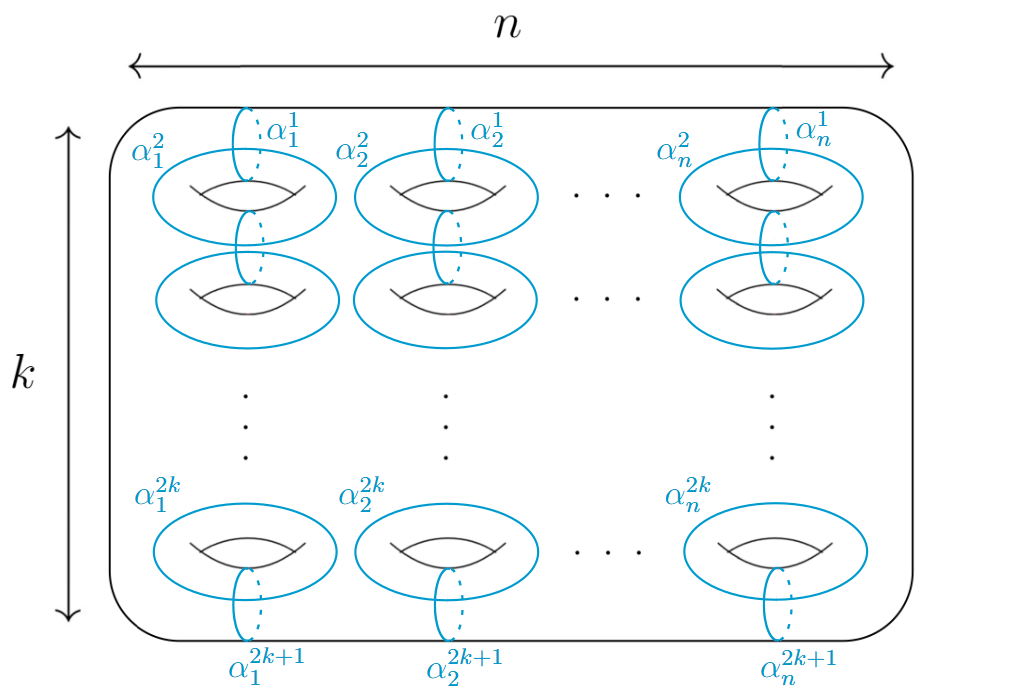}  
  \label{Figure P}
\end{center}

\newpage

Let $\Theta_n^k$ denote the subgroup of $\mathrm{Mod}(S)$ generated by every mapping class of the form 
\begingroup
\setequationspacing
\[
f_i\coloneq T_{\alpha_1^i}^{-1}\circ T_{\alpha_2^i}\circ T_{\alpha_3^i}^{-1}...\circ T_{\alpha_n^i}^{(-1)^n}, 
\]
\endgroup
for $i\in\{1,...,2k+1\}$. Our first objective is to determine the structure of the group $\Theta_n^k$ for each $n,k\in\mathbb{N}$, and our second is to express hyperelliptic involutions in terms of its generators $f_i$. To this effect, the following theorems are proven.\\
\\\textbf{Theorem 2.1.} \textit{We have that $\Theta_2^1\cong \mathrm{SL}_2(\mathbb{Z})$.} (\blue{Section \ref{Section 3}})
\\
\\This will be shown using the symplectic representation $\Psi$ of mapping class groups. Note, it is a standard result that $\mathrm{Mod}(S_1)\cong\mathrm{SL}_2(\mathbb{Z})$, and thus according to this theorem, $\Theta_2^1$ acts as a copy of $\mathrm{Mod}(S_1)$ inside $\mathrm{Mod}(S_2)$. Notice also that $\Theta_1^1=\mathrm{Mod}(S_1)$ by construction, so that $\Theta_1^1\cong\Theta_2^1$. This next theorem shows that whenever $k,n\in\mathbb{N}$ satisfy $n\geq 3$ and $S$ is the closed surface of genus $nk$, there is a copy of the braid group on $2k+2$ strands (here denoted $B_{2k+2}$) modulo its center (here denoted $Z(B_{2k+2}$)) inside $\mathrm{Mod}(S)$.
\\
\\\textbf{Theorem 3.3.}{\spaceskip=4pt
\textit{ Let $k,n\in\mathbb{N}$ with $n\geq3$. Then, $\Theta_n^k\cong B_{2k+2}/Z(B_{2k+2})$.} (\blue{Section \ref{Section 5.3}})}\\
\\ The proof for the above result requires a lot of machinery to set up. In particular, we go over cutting, capping and gluing homomorphisms in \blue{Section \ref{Section 4.1}} and the Birman-Hilden theorem in \blue{ Section \ref{Section 4.2}}. Since the center of any standard braid group is isomorphic to $\mathbb{Z}$, the above theorem implies that we have a short exact sequence 
\begingroup
\setequationspacing
\[
1\longrightarrow \mathbb{Z} \longrightarrow B_{2k+2} \longrightarrow \Theta^k_n \longrightarrow 1
\]
\endgroup
for every $k,n\in\mathbb{N}$ with $n\geq 3$. Furthermore, we know that the inner autormorphism {\spaceskip=3pt group of a group $G$, denoted $\mathrm{Inn}(G)$, is always isomorphic to $G$ modulo its center.} Hence, by the above theorem we have an embedding of $\mathrm{Inn}(B_{2k+2})$ inside $\mathrm{Mod}(S)$.

As for the remaining three theorems, we focus our attention to the case where $k=1$ and relabel the generators of $\Theta_n^1$ to be $a\coloneq f_1$, $b\coloneq f_2$, and $c\coloneq f_3$. It will be shown that $a=c$ whenever the genus is 1 or 2.
\\
\\\textbf{Theorem 4.1.} \textit{Let $S$ be a closed surface of genus $n$. We have that $(abc)^2\in\mathrm{Mod}(S)$ is a hyperelliptic involution on $S$.} (\blue{Section \ref{Section 5}})\\
\\Recall that the hyperelliptic involution is unique for genus 1 and 2, and that all hyperelliptic involutions are conjugate for genus 3 and higher \cite[Section 7.4]{Primer}, so in particular they are all conjugate to $(abc)^2$ according to this theorem.

\newpage

Assume $n\geq 3$ so that we have $\Theta_n^1\cong B_4/Z(B_4)$. One may notice that the hyperelliptic involution $(abc)^2\in\Theta_n^1$ corresponds exactly to the half-twist braid inside $B_4$:

\vspace{8pt}

\noindent\begin{minipage}{\textwidth}
\centering
\includegraphics[scale=0.42]{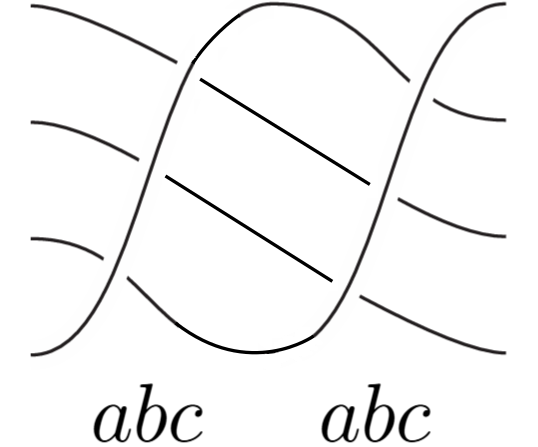}
\end{minipage}

\vspace{9pt}

In this illustration, $a,b,$ and $c$ correspond to bottom, middle and top crossings respectively. It is interesting to note that this is a square root of the full-twist braid $(abc)^4$, which generates the center of $B_4$ cyclically, and by which we are modding out. This is discussed further in \blue{Section \ref{Section 5}}. These last two theorems concern the roots of hyperelliptic involutions.
\\
\\\textbf{Theorem 5.2.} \textit{Let $S$ be a closed surface of genus $n\in\{1,2\}$. Then, the hyperelliptic involution inside $\mathrm{Mod}(S)$ has infinitely many square and cubic roots inside $\Theta_n^1$. The former roots are precisely the conjugacy classes of $(aba)$ and $(aba)^{-1}$ within $\Theta_n^1$, and the latter roots are precisely the conjugacy classes of $(ab)$ and $(ab)^{-1}$ within $\Theta_n^1$, as well as the hyperelliptic involution itself.} (\blue{Section \ref{Section 5.1}})\\
\\ Since $\Theta_n^1\cong\mathrm{SL}_2(\mathbb{Z})$ 
whenever $n\in\{1,2\}$, the representative elements $(aba),(aba)^{-1},(ab),$ and $(ab)^{-1}$ mentioned in the above theorem may be identified with matrices. Respectively, these will turn out to be the elliptic matrices
\vspace{2pt}
\begingroup
\setequationspacing
\[
\begin{bmatrix}
   0 & 1\\
   -1 & 0
\end{bmatrix},\quad
\begin{bmatrix}
     0 & -1\\
   1 & 0
\end{bmatrix},\quad
\begin{bmatrix}
     0 & 1\\
   -1 & 1
\end{bmatrix},\quad
\begin{bmatrix}
     1 & -1\\
   1 & 0
\end{bmatrix}.
\]
\endgroup

\vspace{3pt}

The theorem is proven via this identification.
\\
\\\textbf{Theorem 5.3.} \textit{ Let $S$ be a closed surface of genus $n\geq 3$. Then, each hyperelliptic involution in $\mathrm{Mod}(S)$ has infinitely many square roots that are conjugate to either $(abc)\in\Theta_n^1$ or $(abc)^{-1}\in\Theta_n^1$.} (\blue{Section \ref{Section 5.2}})
\\
\\
Finally, in \blue{Section \ref{Section 7}}, we go over the classification of every group of type $\Theta_n^k$, which may be visualized efficiently by partitioning the $\mathbb{N}^2$ lattice. in the following figure, each point $(n,k)$ corresponds to the group $\Theta_n^k$. The group they are isomorphic to depends on the box they lie in.
\newpage
\begin{center}  
  \captionof{figure}{The classification of groups of type $\Theta_n^k$.}
  \includegraphics[width=0.7\textwidth]{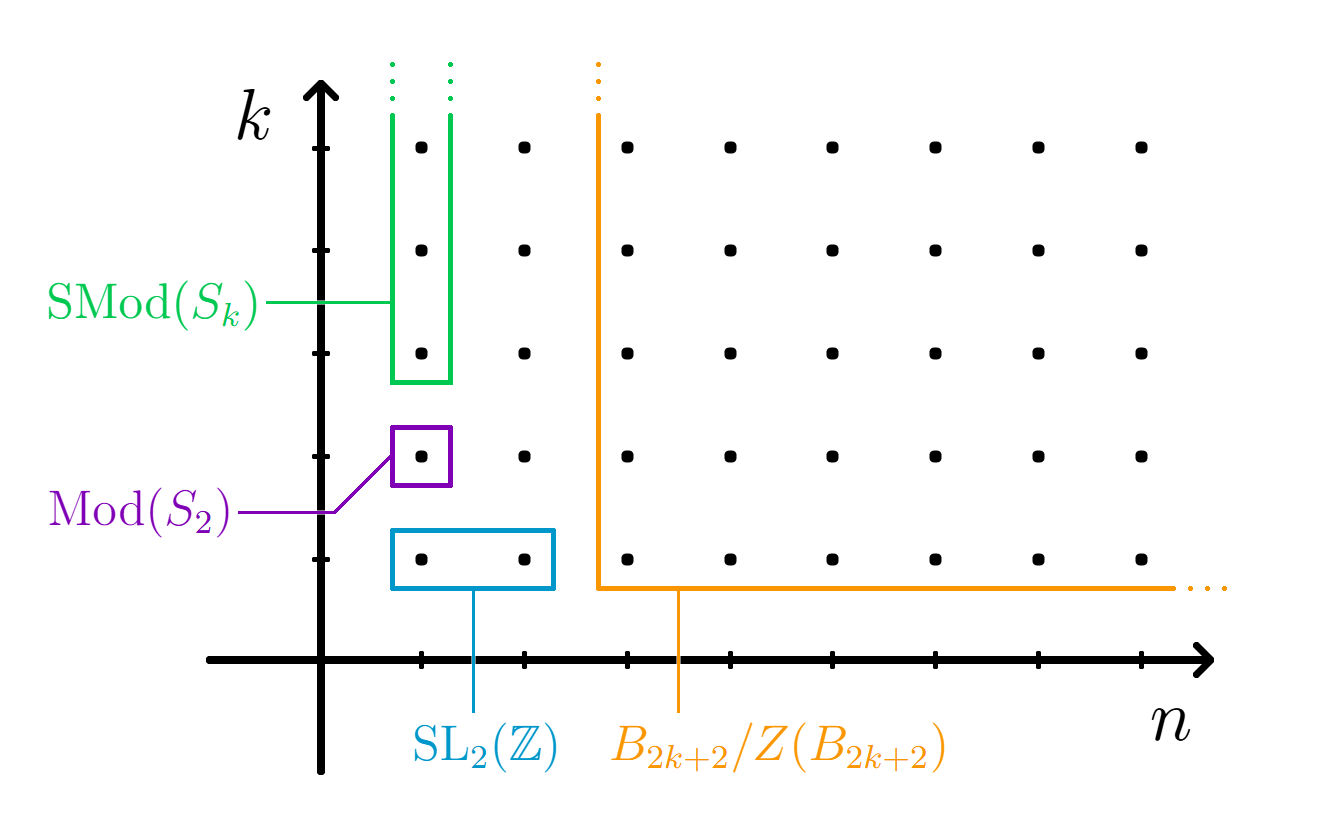}  
  
\end{center}

The blue box is due to \textbf{Theorem 2.1} and the orange box is due to \textbf{Theorem 3.3}. The rest of the classification is discussed in \blue{Section \ref{Section 7}}, and in particular, we will show that the groups of type $\Theta_2^k$ with $k\geq 2$ don't belong in the orange or green boxes. Their exact structure, however, is left as an unsolved problem.

\vspace{15pt}

\textbf{\textit{Acknowledgements.}} The author would like to thank the following.
\begin{itemize}
    \item The EPSRC, for their financial support (grant number: EP/T517896/1).
    \item Philipp Bader, for supervising the internship which led to the writing of this paper and for his extensive help.
    \item Jim Belk, Tara Brendle, and John Nicholson for helpful conversations.
\end{itemize}

\vspace{15pt}

\section{Finding a copy of $\mathrm{Mod}(S_1)$ inside $\mathrm{Mod}(S_2)$}

\label{Section 3}

It is a standard result that the symplectic representation of $\mathrm{Mod}(S_1)$ yields an isomorphism to the special linear group $\mathrm{SL}_2(\mathbb{Z})$ \cite[Theorem 2.5]{Primer}. We may therefore make use of the group presentation in \cite[Section 5.1.3]{Primer}:
\begingroup
\setequationspacing
\begin{align}
\mathrm{Mod}(S_1)&\cong\mathrm{SL}_2(\mathbb{Z})=\langle\,A,B\,|\,ABA=BAB,\,\,(AB)^6=1\,\rangle, \\
&\text{where  } A=\begin{bmatrix}
1 & 1\\
0 & 1
\end{bmatrix} \text{  and  } B=\begin{bmatrix}
1 & 0\\
-1 & 1
\end{bmatrix}. \nonumber
\end{align}
\endgroup
Here, $A$ and $B$ correspond to Dehn twists about a meridian curve and a longitudinal curve on the torus respectively. In this section, it will be shown that there is a subgroup $\mathcal{T}$ of $\mathrm{Mod}(S_2)$ whose symplectic representation yields an analogous isomorphism.

\newpage

\begin{center}  
  \captionof{figure}{}
  \includegraphics[width=0.53\textwidth]{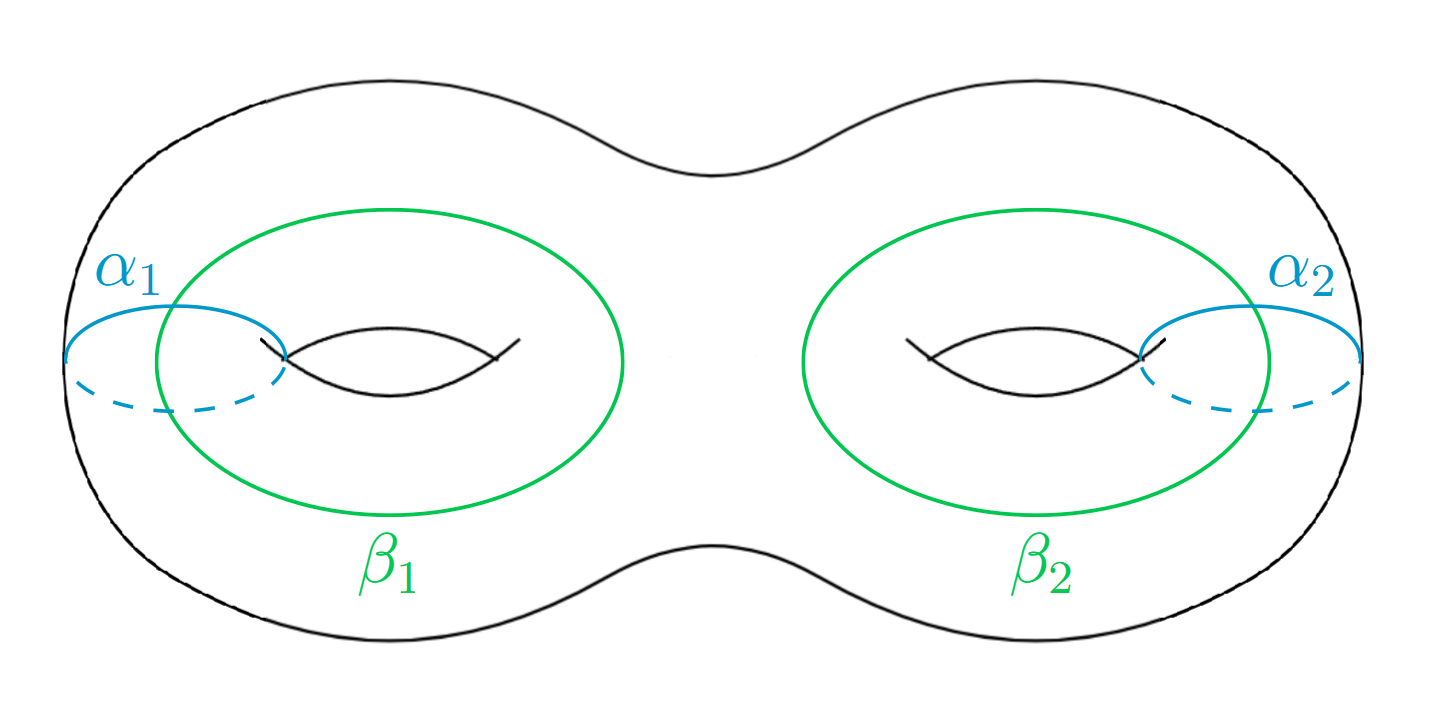}  
  \label{Figure OG}
\end{center}

First, consider the curves on $S_2$ illustrated in the above picture and define the following mapping classes in $\mathrm{Mod}(S_2):$
\begin{align*}
a&\coloneq T_{\alpha_1}^{-1}\circ T_{\alpha_2},\\
b&\coloneq T_{\beta_1}^{-1}\circ T_{\beta_2}.
\end{align*}

Let $\mathcal{T}\coloneq\langle a,b\rangle\leq\mathrm{Mod}(S_2)$. The following theorem shows that $\mathcal{T}$ essentially acts as a copy of $\mathrm{Mod}(S_1)$ inside $\mathrm{Mod}(S_2)$.

\begin{theorem}
\label{Theorem 3.1}
    We have that $\mathcal{T}\cong \mathrm{SL}_2(\mathbb{Z})$.
\end{theorem}
\begin{proof}
We will show that the symplectic representation $\Psi$ of $\mathrm{Mod}(S_2)$ restricts to an isomorphism on $\mathcal{T}$. First, we determine the image of $\mathcal{T}$ under $\Psi$. Notice that
\[
\Psi(a)=
\left[
\begin{array}{c|c}
A^{-1} & \boldsymbol{0} \\
\hline
\boldsymbol{0} & A
\end{array}
\right], \text{  and  }\Psi(b)=
\left[
\begin{array}{c|c}
B^{-1} & \boldsymbol{0} \\
\hline
\boldsymbol{0} & B
\end{array}
\right],  
\]
so it is easily verified that $\langle\Psi(a),\Psi(b)\rangle\cong\langle A,B\rangle=\mathrm{SL}_2(\mathbb{Z})$. Since $a$ and $b$ generate $\mathcal{T}$, we conclude that the image of $\mathcal{T}$ under $\Psi$ is a subgroup of $\mathrm{Sp}_4(\mathbb{Z})$ which is isomorphic to $\mathrm{SL}_2(\mathbb{Z})$. 

Next, in the interest of proving the injectivity of this restriction, we would like to show that every relation in the presentation (2.1) is satisfied by the generators of $\mathcal{T}$. We begin by proving $(ab)^6=1$. Define the following notational shorthands:
\[
a_1=T_{\alpha_1}^{-1},\,\,\,a_2=T_{\alpha_2},\,\,\,b_1=T_{\beta_1}^{-1},\,\,\,b_2=T_{\beta_2}.
\]
Since these elements are individual Dehn twists in $\mathrm{Mod}(S_2)$, they satisfy the disjointness relations \cite[Fact 3.9]{Primer}. So in particular we may write
\[
ab=(a_1a_2)(b_1b_2)=(a_1b_1)(a_2b_2),
\]
and, again by disjointness,
\[
(ab)^6=((a_1b_1)(a_2b_2))^6=(a_1b_1)^6(a_2b_2)^6.
\]
But then, by a chain relation \cite[Section 4.4.1]{Primer},
\begin{align}
(ab)^6=T_{\boldsymbol{\varepsilon}}\circ T_{\boldsymbol{\varepsilon}}^{-1}=1,
\end{align}
where $\boldsymbol{\varepsilon}$ is the following curve in $S_2$:

\begin{center}  
  \includegraphics[width=0.6\textwidth]{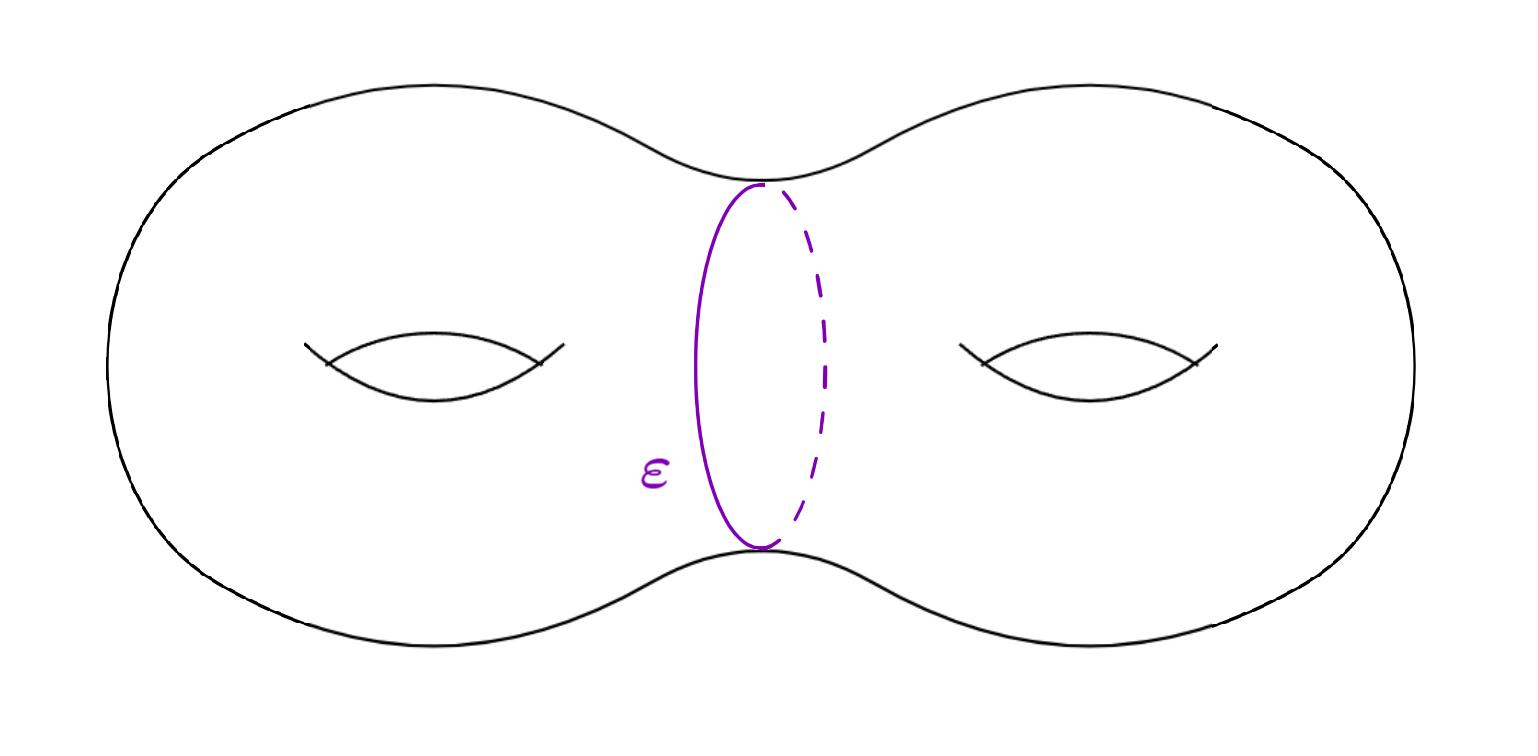}
\end{center}

It remains to show that the braid relation $aba=bab$ holds. We already know that the braid and disjointness relations are satisfied by the Dehn twists $a_1$, $a_2$, $b_1$ and $b_2$ \cite[Section 3.5.1]{Primer}. Therefore,
\begin{align}
    aba&=(a_1a_2)(b_1b_2)(a_1a_2)=a_1b_1a_2a_1b_2a_2=(a_1b_1a_1)(a_2b_2a_2) \nonumber\\
    &=(b_1a_1b_1)(b_2a_2b_2)=b_1a_1b_2b_1a_2b_2=(b_1b_2)(a_1a_2)(b_1b_2)=bab.
\end{align}
Now, suppose $w$ is a word in $\mathcal{T}$ written in the alphabet $\{a,b\}$ such that $\Psi(w)=1$. This means $\Psi(w)$ is a product of conjugates of the relators in the presentation (2.1), spelt in the alphabet $\{A,B\}$. But then $w$ is also a product of conjugates of these relators, spelt in the alphabet $\{a,b\}$. Thus, the equations (2.2) and (2.3) imply $w=1$. We conclude $\Psi$ restricts to an isomorphism on $\mathcal{T}$, which means $\mathcal{T}\cong\mathrm{SL}_2(\mathbb{Z})$.
\end{proof}

\vspace{10pt}

\section{Finding a copy of $B_{2k+2}/Z(B_{2k+2})$ inside $\mathrm{Mod}(S)$}
\label{Section 4}
Let $k,n\in \mathbb{N}$
and let $S$ denote the closed surface of genus $nk$. Recall the simple closed curves on $S$ illustrated in \blue{Figure \ref{Figure P}}, and recall that $\Theta_n^k$ denotes the subgroup of $\mathrm{Mod}(S)$ generated by every mapping class of the form 
\[
f_i\coloneq T_{\alpha_1^i}^{-1}\circ T_{\alpha_2^i}\circ T_{\alpha_3^i}^{-1}...\circ T_{\alpha_n^i}^{(-1)^n}, 
\]
for $i\in\{1,...,2k+1\}$.

Let us first study the special case where $k=1$ and $n=2$. Recall the curves $\alpha_1$ and $\alpha_2$ illustrated in \blue{Figure \ref{Figure OG}}. Notice that the curves $\alpha_1^1$ and $\alpha_1^3$ are both isotopic to the curve $\alpha_1$, and the curves $\alpha_2^1$ and $\alpha_2^3$ are both isotopic to the curve $\alpha_2$. Hence, in this case, the generators $f_1$ and $f_3$ are equal and correspond to the mapping class $a$ defined in \blue{Section \ref{Section 3}}. Furthermore, the curves $\alpha_1^2$ and $\alpha_2^2$ are equal to the curves $\beta_1$ and $\beta_2$ from \blue{Figure \ref{Figure OG}} respectively, so that $f_2$ corresponds to the generator $b$. As a result, $\Theta_2^1=\mathcal{T}$.

Also, $\Theta_1^1$ is clearly equal to $\mathrm{Mod}(S_1)$. According to \blue{Theorem \ref{Theorem 3.1}}, this must mean that $\Theta_1^1\cong\Theta_2^1\cong\mathrm{SL}_2(\mathbb{Z})$. In this section however, we are interested in finding an explicit group presentation for $\Theta_n^k$ whenever $n\geq3$ and $k$ is arbitrary. This will require the Birman-Hilden theorem \cite[Section 9.4.1]{Primer}, as well as some cutting, capping and gluing homomorphisms, which we discuss below. For a summary of the classification of all groups of type $\Theta_n^k$, see \blue{Section \ref{Section 7}}.

\newpage

\begin{minipage}{\textwidth}
\subsection{Cutting, capping and gluing homomorphisms}
\label{Section 4.1}
Let $n\geq3$ and consider the following simple closed curves on $S$:

\vspace{-2pt}

\begin{center}  
  \captionof{figure}{}  \includegraphics[width=0.6\textwidth]{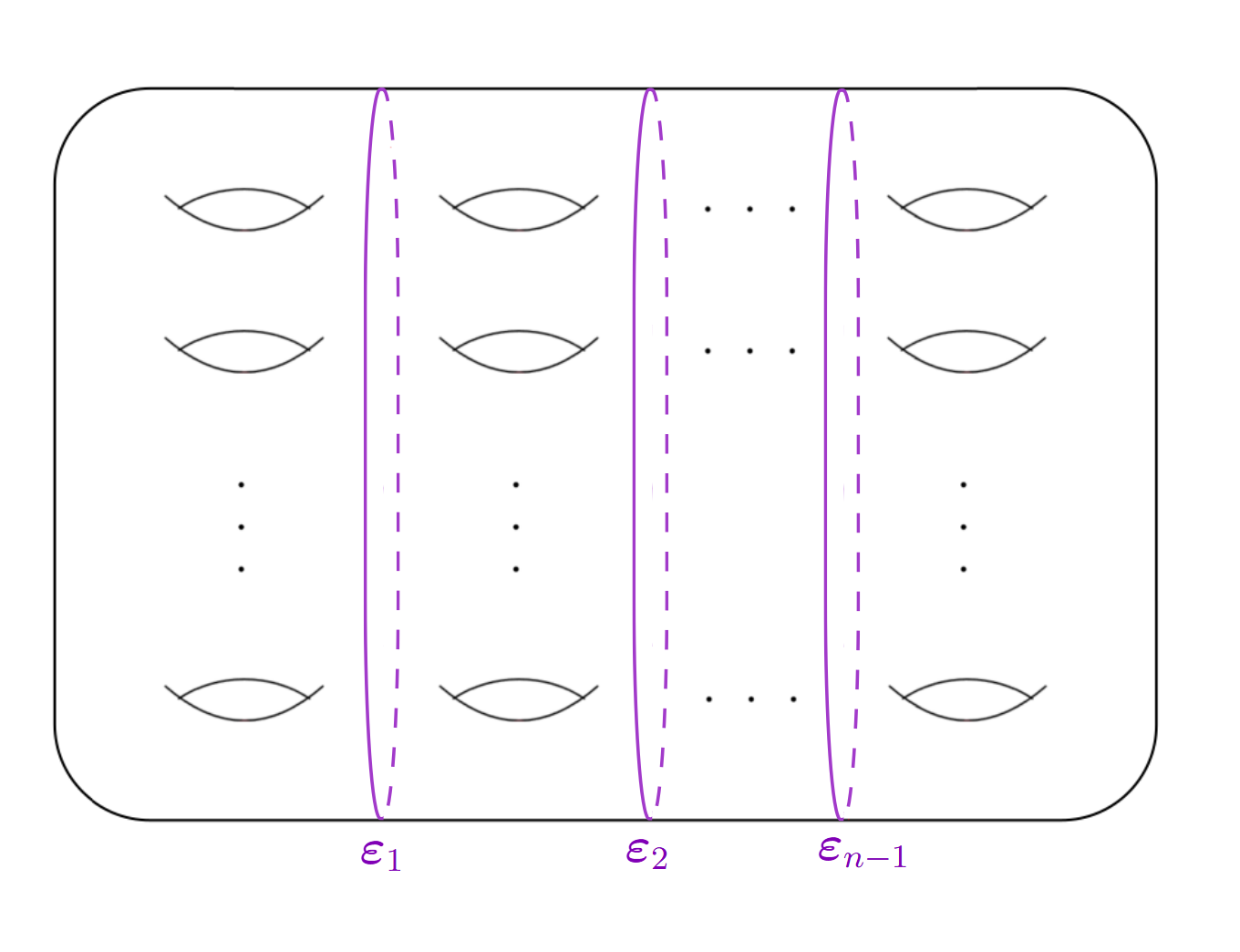}  
  \label{Figure S}
\end{center}

\vspace{-3pt}

If we cut along the curves $\boldsymbol{\varepsilon}_j$ but retain boundary components along the affected areas, then we obtain a surface $S_{\mathrm{cut}}$ which we may express as a disjoint union of surfaces
\begingroup
\setequationspacing
\[
S_{\mathrm{cut}}=S_k^1\sqcup\left(\bigsqcup_{l=1}^{n-2}(S_k^2)_{_l}\right)\sqcup S_k^1,
\]
\endgroup
where $S_k^b$ denotes the surface of genus $k$ with $b$ boundary components. Here we label the connected components of $S_{\mathrm{cut}}$ with a specific ordering in mind:

\vspace{10pt}

\begin{center}  
  \includegraphics[width=0.62\textwidth]{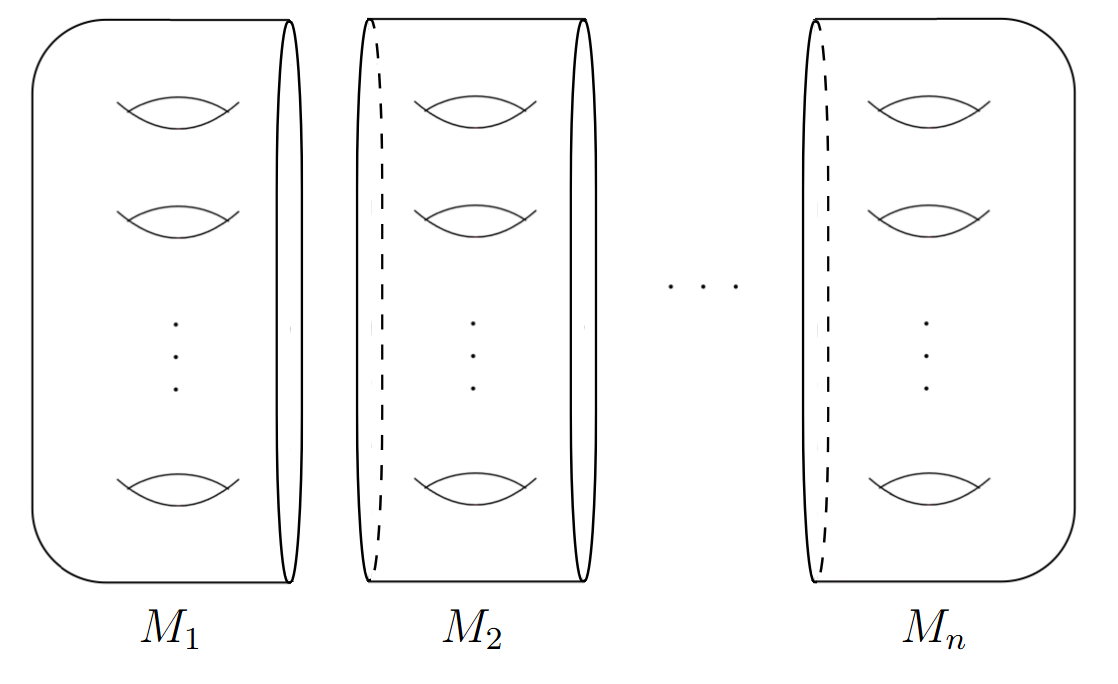}  
\end{center}

Additionally, if we let $S\mkern-1mu{\vrule width0pt height 1.7ex}^\mathrm{\,o}_{\mathrm{cut}}$ denote the interior of $S_{\mathrm{cut}}$, we have that $S\mkern-1mu{\vrule width0pt height 1.7ex}^\mathrm{\,o}_{\mathrm{cut}} \cong S\,\backslash\cup\boldsymbol{\varepsilon}_j$. Notice, moreover, that if we glue a punctured disk onto each boundary component of $S_{\mathrm{cut}}$, we obtain a surface which is homeomorphic to $S\mkern-1mu{\vrule width0pt height 1.7ex}^\mathrm{\,o}_{\mathrm{cut}}$. In other words, we may write
\[
S\mkern-1mu{\vrule width0pt height 1.7ex}^\mathrm{\,o}_{\mathrm{cut}}=S_{k,1}\sqcup\left(\bigsqcup_{l=1}^{n-2}(S_{k,2})_{_l}\right)\sqcup S_{k,1},
\]
where $S_{k,p}$ denotes the surface of genus $k$ with $p$ punctures.
\end{minipage}

\newpage

\begin{center}  
  \includegraphics[width=0.63\textwidth]{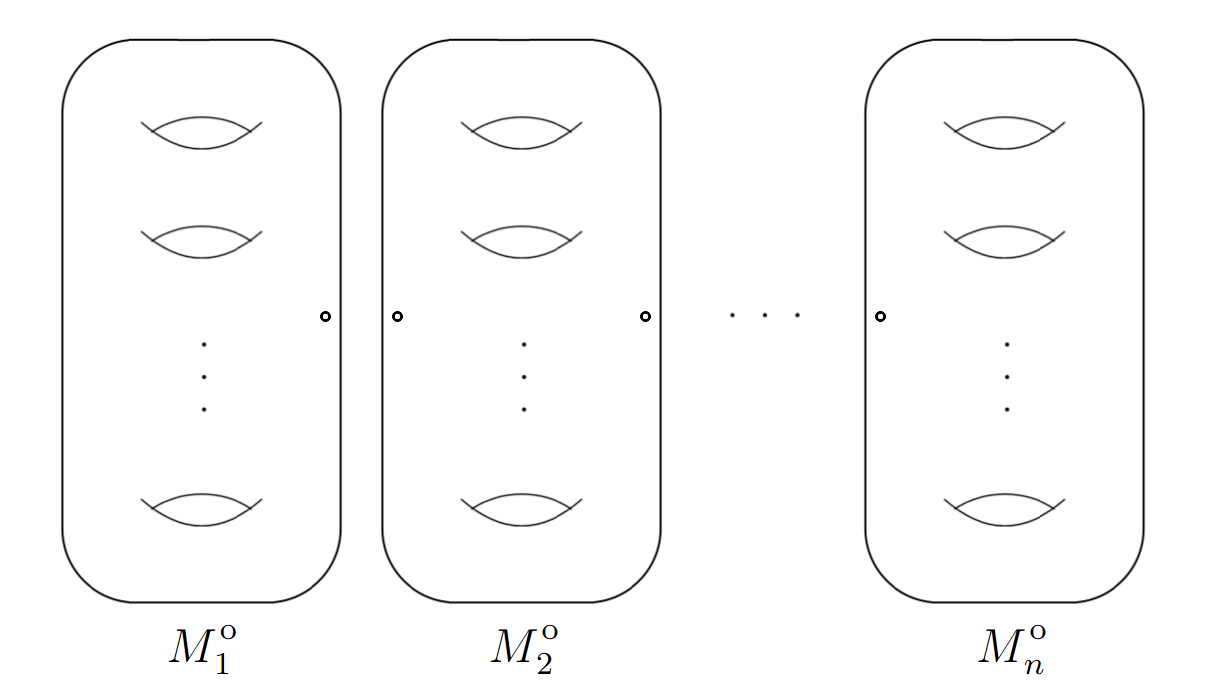}  
\end{center}

\vspace{-2pt}

Now, let $\mathrm{FMod}(S\mkern-1mu{\vrule width0pt height 1.7ex}^\mathrm{\,o}_{\mathrm{cut}})$ denote the mapping classes on $S\mkern-1mu{\vrule width0pt height 1.7ex}^\mathrm{\,o}_{\mathrm{cut}}$ which fix all punctures. Then, we have a well-defined induced homomorphism $\mathfrak{C}:\mathrm{Mod}(S_{\mathrm{cut}})\longrightarrow\mathrm{FMod}(S\mkern-1mu{\vrule width0pt height 1.7ex}^\mathrm{\,o}_{\mathrm{cut}})$ called the capping homomorphism \cite[Section 3.6.2]{Primer}. It is defined in the following way. Suppose $\phi$ is a homeomorphism on $S_{\mathrm{cut}}$. Define $\phi^c$ to be the homeomorphism on $S\mkern-1mu{\vrule width0pt height 1.7ex}^\mathrm{\,o}_{\mathrm{cut}}$ obtained by performing $\phi$ on $S_{\mathrm{cut}}$, capping each boundary component with a punctured disk, and then extending by the identity on said punctured disks. The capping homomorphism is then given by $\mathfrak{C}([\phi])=[\phi^c]$. Technically, this is a composition of several capping homomorphisms (one for each punctured disk), but for the sake of simplicity we will consider it as one. 

Next, we let $\mathrm{FMod}(S)$ denote the mapping classes on $S$ which fix the homotopy class of each curve $\boldsymbol{\varepsilon}_j$. Then, we also have a well-defined induced homomorphism $\mathfrak{K}:\mathrm{FMod}(S)\longrightarrow\mathrm{FMod}(S\mkern-1mu{\vrule width0pt height 1.7ex}^\mathrm{\,o}_{\mathrm{cut}})$ called the cutting homomorphism \cite[Section 3.6.3]{Primer}, which is defined in the following way. For any mapping class $f\in\mathrm{FMod}(S)$, pick a representative homeomorphism $\phi$ of $f$ which fixes $\cup\boldsymbol{\varepsilon}_j$. $\mathfrak{K}(f)$ is simply the mapping class of the restriction of $\phi$ to $S\,\backslash\cup\boldsymbol{\varepsilon}_j$.

Finally, we construct a third induced map $\mathfrak{G}:\mathrm{Mod}(S_{\mathrm{cut}})\longrightarrow\mathrm{Mod}(S)$ which we call the gluing homomorphism. It is defined as follows. Let $\phi$ be a homeomorphism on $S_{\mathrm{cut}}$. Define $\phi'$ to be the homeomorphism on $S$ obtained by performing $\phi$ on $S_{\mathrm{cut}}$, and then gluing the rightmost boundary component of $M_j$ to the leftmost boundary component of $M_{j+1}$ for all $j\in\{1,...,n-1\}$. The gluing homomorphism is then given by $\mathfrak{G}([\phi])=[\phi']$.

We will prove that $\mathfrak{G}$ is well-defined. Assume $\phi$ and $\psi$ are isotopic homeomorphisms on $S_{\mathrm{cut}}$. Then, since $S_{\mathrm{cut}}$ is a disjoint union of $n$ connected surfaces labeled $M_1$, $M_2$, ..., $M_n$, we may express $\phi$ and $\psi$ as two $n$-tuples of homeomorphisms $(\phi_1,...,\phi_n)$ and $(\psi_1,...,\psi_n)$ respectively. Thus, for each $j\in\{1,...,n\}$, there is an isotopy $I_j$ occuring on $M_j$ from $\phi_j$ to $\psi_j$ which fixes the boundary of $M_j$. Now notice that we may denote the connected components in $S\,\backslash\cup\boldsymbol{\varepsilon}_j$ by $M\mkern-1mu{\vrule width0pt height 1.7ex}^\mathrm{\,o}_j$ for any $j\in\{1,...,n\}$. With this in mind, let $I\mkern-1mu{\vrule width0pt height 1.7ex}^\mathrm{\,o}_j$ be the restriction of the isotopy $I_j$ to $M\mkern-1mu{\vrule width0pt height 1.7ex}^\mathrm{\,o}_j$. Next, recall the definition of $\phi'$ from the above paragraph and define $\psi'$ similarly. We may construct an isotopy from $\phi'$ to $\psi'$ as follows: on each $M\mkern-1mu{\vrule width0pt height 1.7ex}^\mathrm{\,o}_j$, perform the isotopy $I\mkern-1mu{\vrule width0pt height 1.7ex}^\mathrm{\,o}_j$, and on $\cup\boldsymbol{\varepsilon}_j$, perform the trivial (constant) isotopy. We can now conclude $[\phi']=[\psi']$, as required.

It is easily verified that $(\phi\circ\psi)'=\phi'\circ\psi'$, which implies $\mathfrak{G}$ is a homomorphism. It should also be noted that the range of $\mathfrak{G}$ lies entirely inside $\mathrm{FMod}(S)$. Indeed, given any homeomorphism $\phi$ on $S_{\mathrm{cut}}$, $\phi$ must fix $\partial (S_{\mathrm{cut}})$ pointwise, and therefore $\phi'$ must fix $\cup\boldsymbol{\varepsilon}_j$ pointwise as required. As a summary of the three maps $\mathfrak{C}$, $\mathfrak{K}$, $\mathfrak{G}$ we've constructed so far, we may now consider the following diagram, which is easily verified to commute:

\vspace{4pt}

\begin{center}
\begin{tikzcd}[row sep=small,every label/.append style={font=\normalfont}]
&\mathrm{FMod}(S) \arrow[dd, "\mathfrak{K}"]\\
  \mathrm{Mod}(S_{\mathrm{cut}})\arrow[ur,"\mathfrak{G}"]\arrow[dr, "\mathfrak{C}"']  \\
 &\mathrm{FMod}(S\mkern-1mu{\vrule width0pt height 1.7ex}^\mathrm{\,o}_{\mathrm{cut}})
\end{tikzcd}
\end{center}
\subsection{The Birman-Hilden theorem}
\label{Section 4.2}
 Let $\iota$ denote the homeomorphism of the surface $S^2_k$ given by the following rotation:
 
\begin{center}          \includegraphics[width=0.60\textwidth]{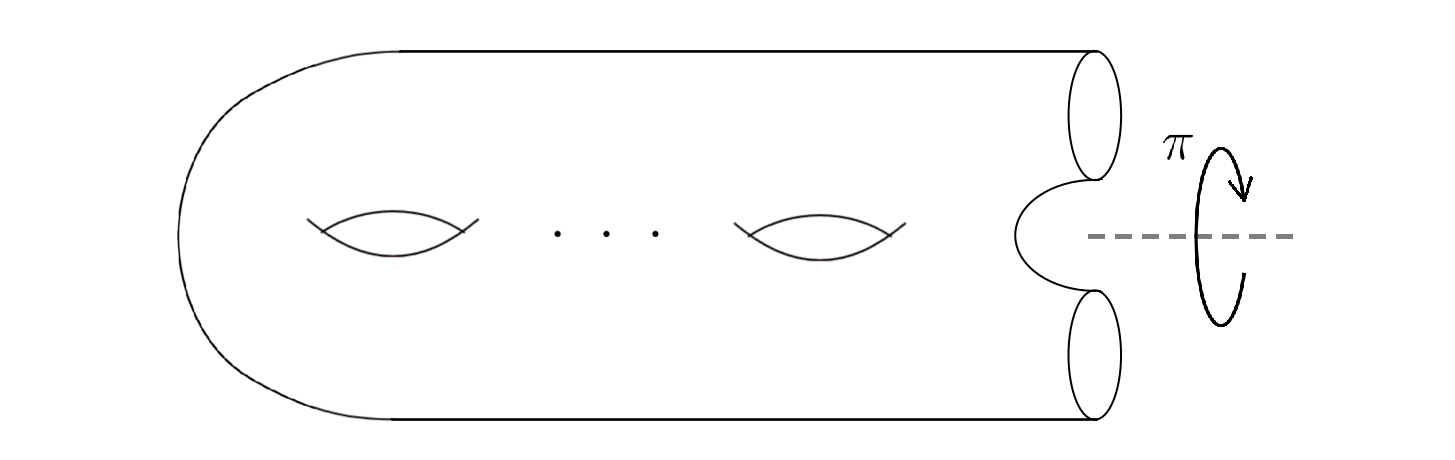}  
\end{center}
 
Note $[\iota]$ is not a mapping class in this case, since $\iota$ does not fix $\partial (S_k^2)$. As a Deck transformation, $\iota$ induces a branched covering $p:S_k^2\longrightarrow D_{2k+2}$ onto the closed disk with $2k+2$ marked points. This is called the Birman-Hilden double cover \cite[Section 9.4]{Primer}. We give the genus 2 example below.

\vspace{2pt}

\begin{center}    \includegraphics[width=0.92\textwidth]{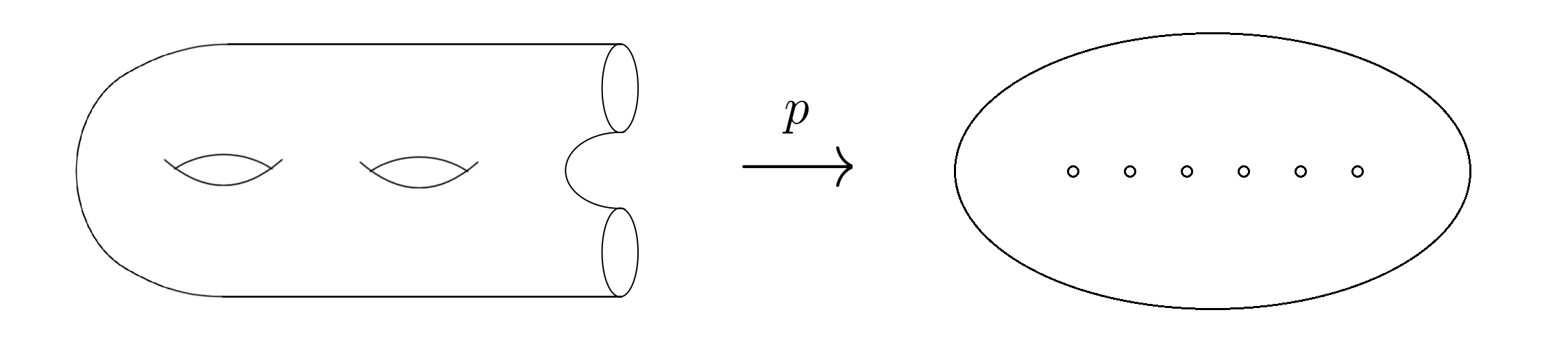}  
\end{center}

The Birman-Hilden double cover simply takes each point in $S^2_k$ to its orbit with respect to the action of $\iota$.
Notice that each marked point on the disk corresponds to a point of intersection between the surface $S_k^2$ and its axis of rotation. These are the branch points of this cover, that is, the points which are not evenly covered by $p$.
For the following theorem, which is originally due to Joan Birman and Hugh Hilden \cite{BirmanHilden}, we let $\mathrm{SMod}(S_k^2)$ denote the symmetric mapping classes on $S$ with respect to $p$ (recall the definition in \blue{Section \ref{Section 1.2}}), and let $B_{2k+2}$ denote the standard braid group on $2k+2$ strands.
\begin{theorem}[the Birman-Hilden theorem for $S_k^2$]
\label{Theorem 4.1}
We have that $\mathrm{SMod}(S_k^2)\cong B_{2k+2}$.
\end{theorem}
\begin{proof}
    See \cite[Section 9.4.3]{Primer}.
\end{proof}

The intuition behind this result is that each generating Dehn twist in $\mathrm{SMod}(S_k^2)$ is the lift of a half-twist on $D_{2k+2}$ under $p$, and each half-twist corresponds to a crossing of two adjacent strands in $B_{2k+2}$. We give the genus 2 example below. For more background on this, see \cite[Section 9.4.1]{Primer} and \cite[Section 18.4]{Office}.

\begin{center}  
  \captionof{figure}{Each half-twist $H_{\gamma_i}$ lifts to a Dehn twist $T_{\overline{\gamma}_i}.$}
  \includegraphics[width=0.91\textwidth]{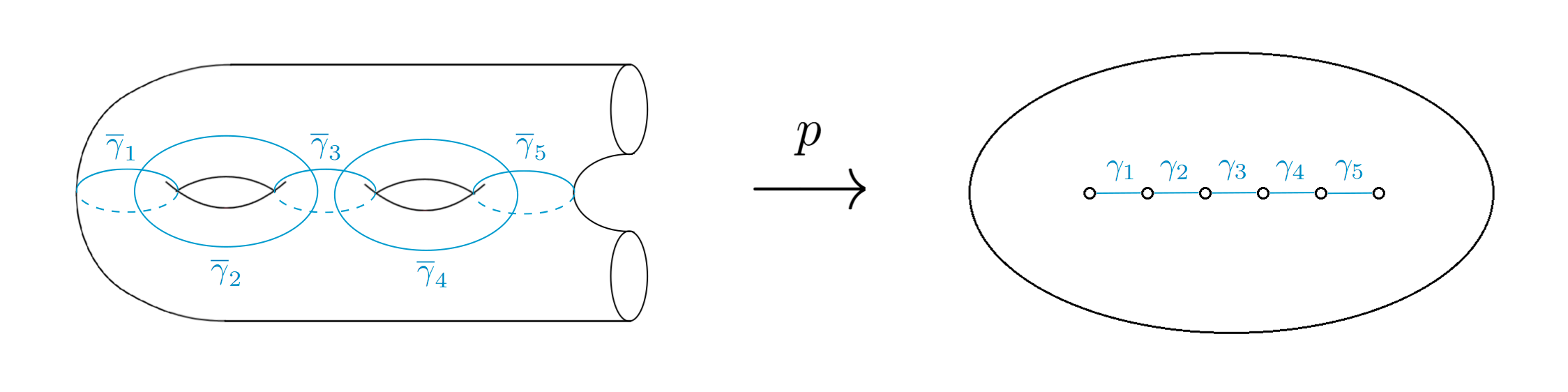}  
  \label{Figure 13}
\end{center}

\vspace{-4pt}

It should be noted that in this case, $\mathrm{SMod}(S_k^2)$ can equivalently be defined as the subgroup of $\mathrm{Mod}(S_k^2)$ consisting of all mapping classes containing a representative homeomorphism which commutes with $\iota$ (see \cite[Section 9.4.1]{Primer}).

From now on, we will write $F_i\coloneq T_{\overline{\gamma}_i}$ for all $i\in\{1,...,2k+1\}$ as a shorthand for the generators of $\mathrm{SMod}(S_k^2)$. By the Birman-Hilden theorem, we may in fact use the standard braid group presentation. Hence the relations in $\mathrm{SMod}(S_k^2)$ are
\begingroup
\setequationspacing
\begin{align}
[F_i,F_j]&=1 \hspace*{.25in}\text{ for } |i-j|>2, \nonumber \\
F_iF_{i+1}F_i=F_{i+1}F_i&F_{i+1} \hspace*{.25in}\text{ for } 1\leq i\leq 2k.
\end{align}
\endgroup

\vspace{7pt}

\subsection{Computing $\Theta_n^k$ for $n\geq3$}
\label{Section 5.3}
We will now combine results from the two previous sections to determine a group presentation for $\Theta_n^k$ whenever $n\geq3$. We will require a lemma (\blue{Lemma \ref{Lemma 4.2}}) which embeds $\mathrm{SMod}(S_k^2)$ inside $\mathrm{Mod}(S_{\mathrm{cut}})$ to make things easier.

Let $i\in\{1,...,2k+1\}$ and recall the definition of the mapping class $f_i\in\mathrm{FMod}(S)$ outlined at the beginning of \blue{Section \ref{Section 4}}. An analogous mapping class $\widehat{f}_i$ can be defined in $\mathrm{Mod}(S_{\mathrm{cut}})$ as follows. Consider the following simple closed curves in $S_{\mathrm{cut}}$:

\begin{center}  
  \includegraphics[width=0.67\textwidth]{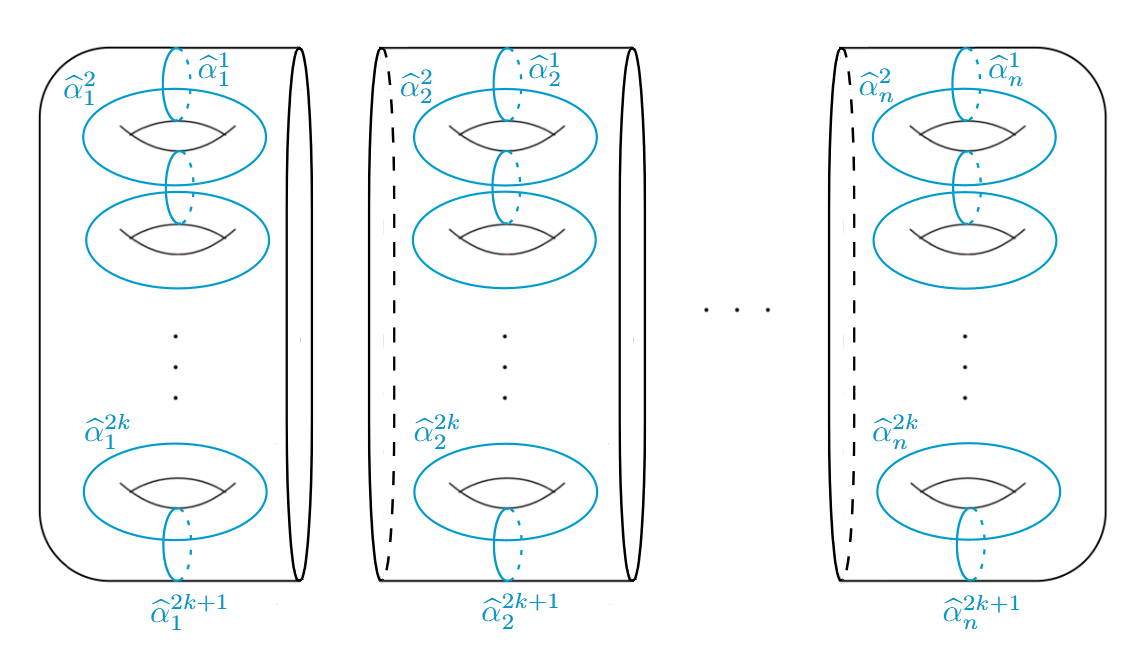}  
\end{center}

Then, define $\widehat{f}_i$ to be the mapping class in $\mathrm{Mod}(S_{\mathrm{cut}})$ given by
\[
\widehat{f}_i\coloneq T_{\widehat{\alpha}_1^{\,i}}^{-1}\circ T_{\widehat{\alpha}_2^{\,i}}\circ T_{\widehat{\alpha}_3^{\,i}}^{-1}...\circ T_{\widehat{\alpha}_n^{\,i}}^{(-1)^n}.
\]
With this in mind, we let $X\coloneq\{F_1,...,F_{2k+1}\}$ be the set of generators of $\mathrm{SMod}(S_k^2)$ from the presentation (3.1), and define $w:X\longrightarrow\mathrm{Mod}(S_{\mathrm{cut}})$ to be the map given by $F_i\longmapsto\widehat{f}_i$.
\begin{lemma}
\label{Lemma 4.2}
    Let $k,n\in\mathbb{N}$ such that $n\geq3$. Then, $w$ extends to an injective homomorphism $\omega:\mathrm{SMod}(S^k_2)~\longrightarrow~\mathrm{Mod}(S_{\mathrm{cut}})$.
\end{lemma}

\begin{proof}
  Let $i\in\{1,...,2k+1\}$ and notice that the curve $\widehat{\alpha}^{\,i}_1$ may be seen as a curve sitting in the surface $S_k^1$. Let $\dot{F_i}$ denote the Dehn twist in $\mathrm{Mod}(S_k^1)$ about this curve. It is clear that $\widehat{f}_i$ can be expressed as a tuple of mapping classes
\[
\widehat{f}_i=(\dot{F}_i^{-1},F_i,F_i^{-1},...,F_i^{(-1)^{n-1}},\dot{F}_i^{(-1)^{n}}).
\]
Certainly, the group $\langle \dot{F}_1,...,\dot{F}_{2k+1}\rangle$ satisfies the disjointness and braid relations specified in the presentation (3.1), just as $\langle F_1,...,F_{2k+1}\rangle$ does. Then, $\langle \widehat{f}_1,...,\widehat{f}_{2k+1}\rangle$ satisfies all the relations that $\mathrm{SMod}(S_k^2)$ does, and therefore, $w$ extends to a well-defined homomorphism $\omega:\mathrm{SMod}(S^k_2)\longrightarrow\mathrm{Mod}(S_{\mathrm{cut}})$.

Now, let $r\in\mathrm{SMod}(S^k_2)$. Notice that $\omega(r)$ is an $n$-tuple of mapping classes which we may write as $(*,*,...,*)$ for convenience, where $*$ denotes a placeholder for each entry. With this in mind, suppose $w(r)=1$. Since $n\geq3$, the second entry in the $n$-tuple $\omega(r)$ is precisely $r$, so we may write this as $(*,r,*,...,*)=(1,1,1,...,1)$, meaning in particular, $r=1$. Hence, $\omega$ is injective. 
\end{proof}

We may now incorporate our new map $\omega$ into the cutting, capping and gluing diagram from \blue{Section \ref{Section 4.1}}:
\begin{center}
\begin{tikzcd}[row sep=small,every label/.append style={font=\normalfont}]
&&\mathrm{FMod}(S) \arrow[dd, "\mathfrak{K}"]\\
\mathrm{SMod}(S_k^2)\arrow[r, hook, "\omega"]  &\mathrm{Mod}(S_{\mathrm{cut}})\arrow[ur,"\mathfrak{G}"]\arrow[dr, "\mathfrak{C}"']  \\
 &&\mathrm{FMod}(S\mkern-1mu{\vrule width0pt height 1.7ex}^\mathrm{\,o}_{\mathrm{cut}})  
\end{tikzcd}
\end{center}

For this final theorem, $Z(B_{2k+2})$ denotes the center of $B_{2k+2}$.
\begin{theorem}
\label{Theorem 4.3}
    Let $k,n\in\mathbb{N}$ with $n\geq3$. Then, $\Theta_n^k\cong B_{2k+2}/Z(B_{2k+2})$.
\end{theorem}
\begin{proof} Notice that $\Theta_n^k$ is a subgroup of $\mathrm{FMod}(S)$, because the mapping class $f_i$ fixes $\cup\boldsymbol{\varepsilon}_j$ pointwise for all $i\in\{1,...,2k+1\}$, so that by extension, any composition of such mapping classes also fixes $\cup\boldsymbol{\varepsilon}_j$ pointwise. Hence, the group $\mathfrak{K}(\Theta_n^k)\leq \mathrm{FMod}(S\mkern-1mu{\vrule width0pt height 1.7ex}^\mathrm{\,o}_{\mathrm{cut}})$ can be considered. We will first show that $\mathfrak{K}(\Theta_n^k)\cong B_{2k+2}/Z(B_{2k+2})$, and then we will show that $\mathfrak{K}$ restricts to an isomorphism on $\Theta_n^k$, thus completing the proof. 

Since $n\geq 3$, we may consider the injection $\omega$ by \blue{Lemma \ref{Lemma 4.2}}. Let $\mathfrak{c}\coloneq\mathfrak{C}\circ\omega$. We will show that the image of $\mathfrak{c}$ is $\mathfrak{K}(\Theta_n^k)$. Indeed, the image of $\omega$ is $\langle \widehat{f}_1,...,\widehat{f}_{2k+1}\rangle~\eqcolon~\widehat{\Theta}_n^k$ by construction, and the diagram above commutes, so 
\[
\mathfrak{C}(\widehat{\Theta}_n^k)=\mathfrak{K}\circ\mathfrak{G}(\widehat{\Theta}_n^k)=\mathfrak{K}(\Theta_n^k).
\]
Next, we would like to determine the kernel of $\mathfrak{c}$, which is equivalent to computing  $\omega^{-1}(\mathrm{ker}(\mathfrak{C}))$. We know from \cite[Section 3.6.2]{Primer} that the kernel of any capping homomorphism is generated by the Dehn twists about each boundary component that gets capped. So in our case, $\mathrm{ker}(\mathfrak{C})$ is generated by Dehn twists about each boundary component of every connected component $M_j$. This group is isomorphic to $\mathbb{Z}^{2n-2}$.

Let $T_1,T_2\in\mathrm{Mod}(S_k^2)$ denote the Dehn twists about the two boundary components of $S^2_k$ respectively. Assume $r\in\omega^{-1}(\mathrm{ker}(\mathfrak{C}))$. Then, $\omega(r)\in\mathrm{ker}(\mathfrak{C})$, so in particular the second entry in the tuple $\omega(r)$ lies inside $\langle T_1, T_2\rangle$. But the second entry in $\omega(r)$ is $r$, hence $r\in\langle T_1, T_2\rangle$. By assumption, we also have $r\in\mathrm{SMod}(S_k^2)$, so we conclude $r\in\mathrm{SMod}(S_k^2)\cap\langle T_1, T_2\rangle\eqcolon N$. Moreover, we clearly have $\omega(N)\subseteq\mathrm{ker}(\mathfrak{C})$, which means $N\subseteq\omega^{-1}(\mathrm{ker}(\mathfrak{C}))$, so we have shown by double inclusion that
\[
N=\omega^{-1}(\mathrm{ker}(\mathfrak{C}))=\mathrm{ker}(\mathfrak{c}).
\]
We now find a better description for $N$. Notice that $\langle T_1, T_2\rangle$ is isomorphic to $\mathbb{Z}^2$ via the identification $T_1\longmapsto (1,0)$ and $T_2\longmapsto (0,1)$. Let $\widehat{T}_1,\widehat{T_2}$ denote the standard representative homeomorphisms of the Dehn twists $T_1$ and $T_2$ respectively. Notice that
\[
\widehat{T}_1\circ\iota=\iota\circ\widehat{T}_2,
\]
which implies that for any $x,y\in\mathbb{Z}$,
\[
\widehat{T}_1^{\,x}\circ\widehat{T}_2^{\,y}\circ\iota=\iota\circ\widehat{T}_1^{\,y}\circ\widehat{T}_2^{\,x}.
\]
Therefore if $x=y$ Then the mapping class $T_1^{\,x}\circ T_2^{\,y}$ has a representative which commutes with $\iota$. However, if $x\neq y$ it is clear that $T_1^{\,x}\circ T_2^{\,y}$ has no representatives which commute with $\iota$. We conclude  $T_1\circ T_2$ generates all the symmetric mapping classes on $S^2_k$ that can be found inside $\langle T_1,T_2\rangle$, or in other words,
    \[
    N\coloneq\mathrm{SMod}(S_k^2)\cap\langle T_1, T_2\rangle=\langle T_1\circ T_2\rangle.
    \]
Furthermore, we know by a chain relation \cite[Section 4.4.1]{Primer} that 
\[
T_1\circ T_2 = (F_1F_2...F_{2k+1})^{2k+2}.
\]
The Birman-Hilden theorem (\blue{Theorem \ref{Theorem 4.1}}) tells us  
that $\mathrm{SMod}(S_k^2)\cong B_{2k+2}$, and with this identification in mind, the element $(F_1F_2...F_{2k+1})^{2k+2}$ actually corresponds to the full-twist braid in $B_{2k+2}$, which generates the center of $B_{2k+2}$ cyclically \cite[Theorem 18.4]{Office}. Thus, the image of $N$ under this identification is $Z(B_{2k+2})$. Now, recall that
\[
\mathfrak{K}(\Theta_n^k)=\mathrm{im}(\mathfrak{c})\cong\mathrm{SMod}(S_k^2)/\mathrm{ker}(\mathfrak{c})=\mathrm{SMod}(S_k^2)/N,
\]
And so we conclude
\[
\mathfrak{K}(\Theta_n^k)\cong B_{2k+2}/Z(B_{2k+2}).
\]
It remains to show that $\mathfrak{K}$ restricts to an isomorphism on $\Theta_n^k$. Let $\mathfrak{g}\coloneq\mathfrak{G}\circ\omega$. By construction, we have that the following diagram commutes.
\vspace{3mm}
\begin{center}
\begin{tikzcd}[row sep=small,every label/.append style={font=\normalfont}]
&&\mathrm{FMod}(S) \arrow[dd, "\mathfrak{K}"]\\
\mathrm{SMod}(S_k^2)\arrow[r, hook, "\omega"] \arrow[urr, bend left=15, "\mathfrak{g}"] \arrow[drr, bend right=15, "\mathfrak{c}"']  &\mathrm{Mod}(S_{\mathrm{cut}})\arrow[ur,"\mathfrak{G}"]\arrow[dr, "\mathfrak{C}"']  \\
 &&\mathrm{FMod}(S\mkern-1mu{\vrule width0pt height 1.7ex}^\mathrm{\,o}_{\mathrm{cut}}) 
\end{tikzcd}
\end{center}
\vspace{4mm}
 Let $T_{\Delta}\coloneq T_1\circ T_2\in N$ and let $\delta\coloneq\partial(S_k^1)$. Notice that
\begin{align}
\omega(T_{\Delta})=(T_{\delta}^{-1},T_{\Delta},T_{\Delta}^{-1},...,T_{\Delta}^{(-1)^{n-1}},T_{\delta}^{(-1)^n}).
\end{align}
We can deduce from this that $\mathfrak{g}(T_{\Delta})=\mathfrak{G}\circ\omega(T_{\Delta})=1$, because when the rightmost boundary component of $M_j$ is glued to the leftmost boundary component of $M_{j+1}$ for all $j\in\{1,...,n-1\}$, the Dehn twists in the tuple (3.2) cancel out with each other. Since $T_{\Delta}$ generates $N$, we conclude $\mathfrak{g}(N)=1$.

We will now show that $N=\mathfrak{g}^{-1}(\mathrm{ker}(\mathfrak{K}))$. Since $1\in\mathrm{ker}(\mathfrak{K})$, we know from the previous paragraph that $N\subseteq\mathfrak{g}^{-1}(\mathrm{ker}(\mathfrak{K}))$. Conversely, assume $r\in\mathfrak{g}^{-1}(\mathrm{ker}(\mathfrak{K}))$. Then $\mathfrak{g}(r)\in\mathrm{ker}(\mathfrak{K})$, but since the above diagram commutes, this means
\[
\mathfrak{c}(r)=\mathfrak{K}\circ\mathfrak{g}(r)=1.
\]
In other words, $r\in\mathrm{ker}(\mathfrak{c})=N$, as required. 

We are now ready for our final maneuver. Write $H\coloneq \mathrm{SMod}(S_k^2)$ for notational convenience. Clearly $\mathfrak{g}(H)=\Theta_n^k$, and so
\[
\mathfrak{g}^{-1}(\mathrm{ker}(\mathfrak{K})\cap\Theta_n^k)=\mathfrak{g}^{-1}(\mathrm{ker}(\mathfrak{K}))\cap\mathfrak{g}^{-1}(\Theta_n^k)=N\cap H=N.
\]
Furthermore, we have that
    \begin{alignat*}{2}
    &&\mathfrak{g}^{-1}(\mathrm{ker}(\mathfrak{K})\cap\Theta_n^k)
    &=N\\
    &\implies\quad
    &\mathfrak{g}\circ\mathfrak{g}^{-1}(\mathrm{ker}(\mathfrak{K})\cap\Theta_n^k)
    &=\mathfrak{g}(N)\\
    &\implies\quad
    &\mathfrak{g}(H)\cap(\mathrm{ker}(\mathfrak{K})\cap\Theta_n^k)
    &=1\\
    &\implies\quad
    &\mathrm{ker}(\mathfrak{K})\cap\Theta_n^k
    &=1.
    \end{alignat*}
This shows that $\mathfrak{K}$ restricts to an isomorphism on $\Theta_n^k$, and the proof is complete.
\end{proof}
We now have a suitable group presentation for $\Theta_n^k$ whenever $n\geq3$: simply add the full-twist braid relator to the presentation (3.1). Our generators are $f_1$,...,$f_{2k+1}$, and the relations are
\begingroup
\setequationspacing
\begin{align*}
[f_i,f_j]&=1 \hspace*{.25in}\text{ for } |i-j|>2,\\
f_if_{i+1}f_i=f_{i+1}f_i&f_{i+1} \hspace*{.25in}\text{ for } 1\leq i\leq 2k,\\
(f_1f_2...f_{2k+1})^{2k+2}&=1.
\end{align*}
\endgroup
Notice that the resulting presentation does not itself depend on $n$. 

One consequence of this theorem is that $\Theta_n^k\cong \mathrm{F_1Mod}(S_{0,2k+3})$, where $\mathrm{F_1Mod}(S_{0,2k+3})$ denotes the group of mapping classes on the punctured sphere $S_{0,2k+3}$ which fix one distinguished puncture. This is verified easily by applying the capping homomorphism to a punctured disk $D_{2k+2}$ and obtaining a punctured sphere $S_{0,2k+3}$ with one extra puncture which remains fixed. The desired result then follows from the fact that $\mathrm{Mod}(D_{2k+2})\cong B_{2k+2}$. See \cite[Section 9.2]{Primer} for more background on this.

Another consequence of this result is that, since the center of any standard braid group is isomorphic to $\mathbb{Z}$, we have the following short exact sequence
\[
1\longrightarrow \mathbb{Z} \longrightarrow B_{2k+2} \longrightarrow \Theta^k_n \longrightarrow 1
\]
for every $k,n\in\mathbb{N}$ with $n\geq 3$. In \blue{Section \ref{Section 7}}, we show that $\Theta_2^k$ is not isomorphic to $\Theta_n^k$ when $n\geq 3$, and further discuss the potential classification of all groups of type $\Theta_n^k$.

\vspace{11pt}

\section{Expressing hyperelliptic involutions in terms of braids}
\label{Section 5}

Suppose $S$ is a closed surface of genus $n$, so that the group $\Theta_n^1$ may be seen as a subgroup of $\mathrm{Mod}(S)$. In this section, we construct a hyperelliptic involution on the surface $S$ by using the generators of $\Theta_n^1$. If we write $a\coloneq f_1$, $b\coloneq f_2$, and $c\coloneq f_3$, then, as proven in the two previous sections, we obtain the following presentations:
\begin{align}
\Theta_n^1=\langle\,a,b,c,\,|\,aba=bab,\,\,cbc=bcb,\,\,ac=ca,\,\,(abc)^4=1\,\rangle\cong B_4/Z(B_4)
\end{align}
for all $n\geq 3$, and
\begin{align}
\Theta_1^1\cong\Theta_2^1=\langle\,a,b\,|\,aba=bab,\,\,(ab)^6=1\,\rangle\cong\mathrm{SL}_2(\mathbb{Z}),
\end{align}
keeping in mind that for surfaces of genus 1 or 2, the generators $a$ and $c$ are equal, which is why the latter is omitted from presentation (4.2).

In the upcoming theorem, we will show that, regardless of genus, the mapping class $(abc)^2\in\mathrm{Mod}(S)$ is always equal to a hyperelliptic involution. For genus 1 and 2, this becomes the element $(aba)^2$ inside $\mathrm{SL}_2(\mathbb{Z})$, which yields the matrix $-\mathrm{Id}_2\in\mathrm{SL}_2(\mathbb{Z})$. The consequences of this are discussed in \blue{Section \ref{Section 5.1}}. As for genus 3 and above, since $\Theta_n^1$ is a braid group modulo its center $\langle(abc)^4\rangle$, the word $(abc)^2$ can be expressed as a braid. Crucially, it becomes the square root of the full-twist braid $(abc)^4$ by which we are modding out. Said full-twist braid can be visualized with the following diagram (taken from \cite[Figure 18.11]{Office}), where $a,b,$ and $c$ correspond to bottom, middle and top crossings respectively.

\begin{center}  
  \includegraphics[width=0.52\textwidth]{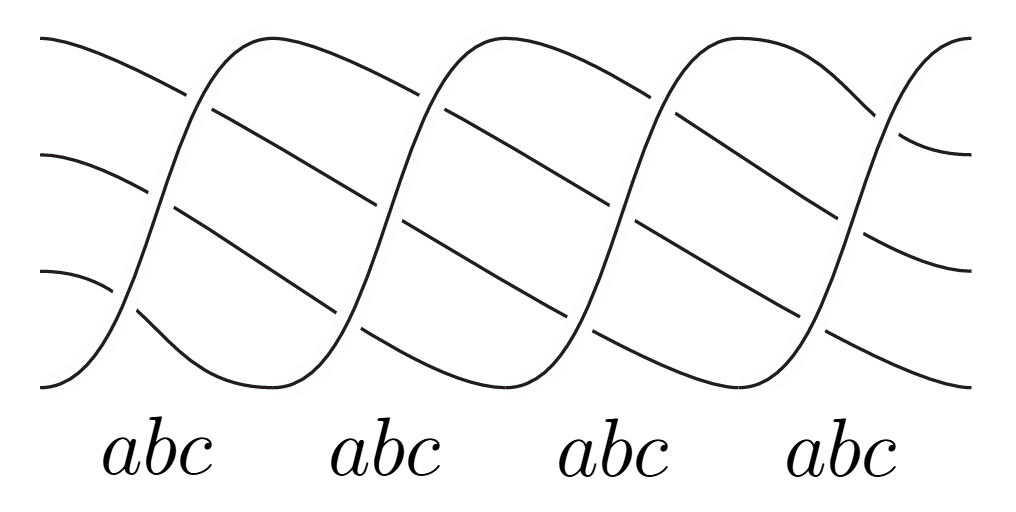}  
\end{center}

\vspace{-3pt}

This is the braid one obtains by taking the identity braid and then spinning the entire right wall by 360 degrees \cite[Section 18.2]{Office}. Thus, the hyperelliptic involution $(abc)^2$ inside $\Theta_n^1$ may be seen as the half-twist braid obtained by doing the same thing but only spinning the wall by 180 degrees. 

\newpage

\vspace{7pt}

\begin{center}  
\includegraphics[width=0.27\textwidth]{images/braid_2.png}  
\end{center}

\vspace{4pt}

Since the conjugation of a hyperelliptic involution always yields a hyperelliptic involution, there are many more distinct braids conjugate to this one which also correspond to hyperelliptic involutions inside $\Theta_n^1$. 

\begin{theorem}
\label{Theorem 5.1}
    Let $S$ be a closed surface of genus $n$. we have that $(abc)^2\in\mathrm{Mod}(S)$ is a hyperelliptic involution.
\end{theorem}

\begin{proof}
We will show that $\tau\coloneq(cab)^2$ is a hyperelliptic involution. This will suffice, because
\[
(abc)^2=c^{-1}(cab)^2c,
\] 
and conjugating a hyperelliptic involution always yields a hyperelliptic involution. By definition, we must show two things:
\begin{enumerate}
\item $\tau^2=1$ and
\item $\Psi(\tau)=-\mathrm{Id}_{2n}\in\mathrm{Sp}_{2n}(\mathbb{Z})$,
\end{enumerate}
where $\Psi$ denotes the symplectic representation of $\mathrm{Mod}(S)$. First, notice that $(abc)^4=1$ for surfaces of genus 3 and higher, according to the presentation (4.1). Furthermore, since $a=c$ whenever the genus is 1 or 2, we may write, in such instances:
\[
(abc)^4=(aba)^4=aba(aba)aba(aba)=aba(bab)aba(bab)=(ab)^6=1,
\]
{\spaceskip=2pt where the last equality holds via the presentation (4.2). Therefore, the~relation~$(abc)^4~=~1$} is true for any genus, so we may now write,
\[
\tau^2=(cab)^4=c(abc)^4c^{-1}=cc^{-1}=1,
\]
and thus the first requirement is proven. To prove the second, recall the standard generators of $H_1(S)$ illustrated below:

\begin{center}    \includegraphics[width=0.96\textwidth]{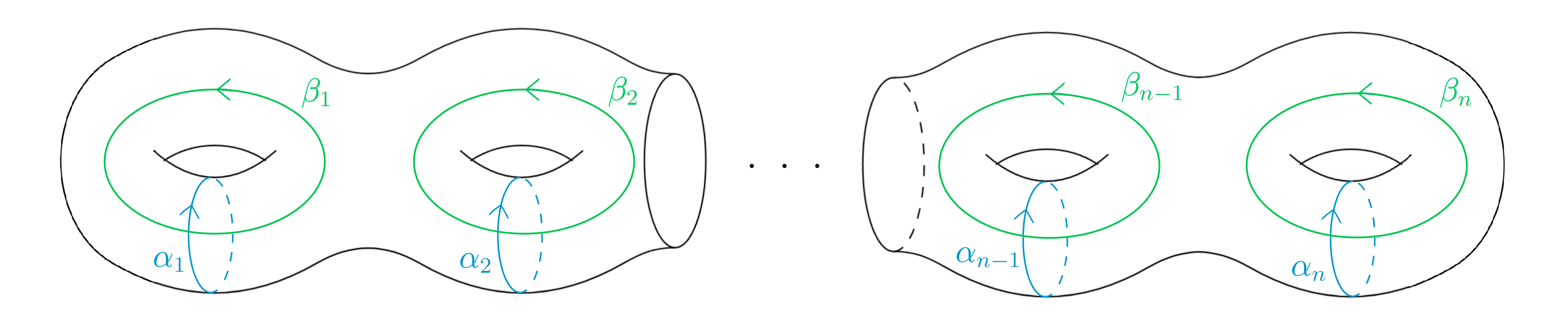}  
\end{center}

Notice that $\tau=cabcab=cabacb=cbabcb$. Suppose $j\in\{1,...,n\}$ is odd. We shall assess how the mapping class $cbabcb$ affects the homological generators $\alpha_j$ and $\beta_j$ individually on the surface $S$:
\begin{center}  
  \captionof{figure}{}
  \includegraphics[width=0.79\textwidth]{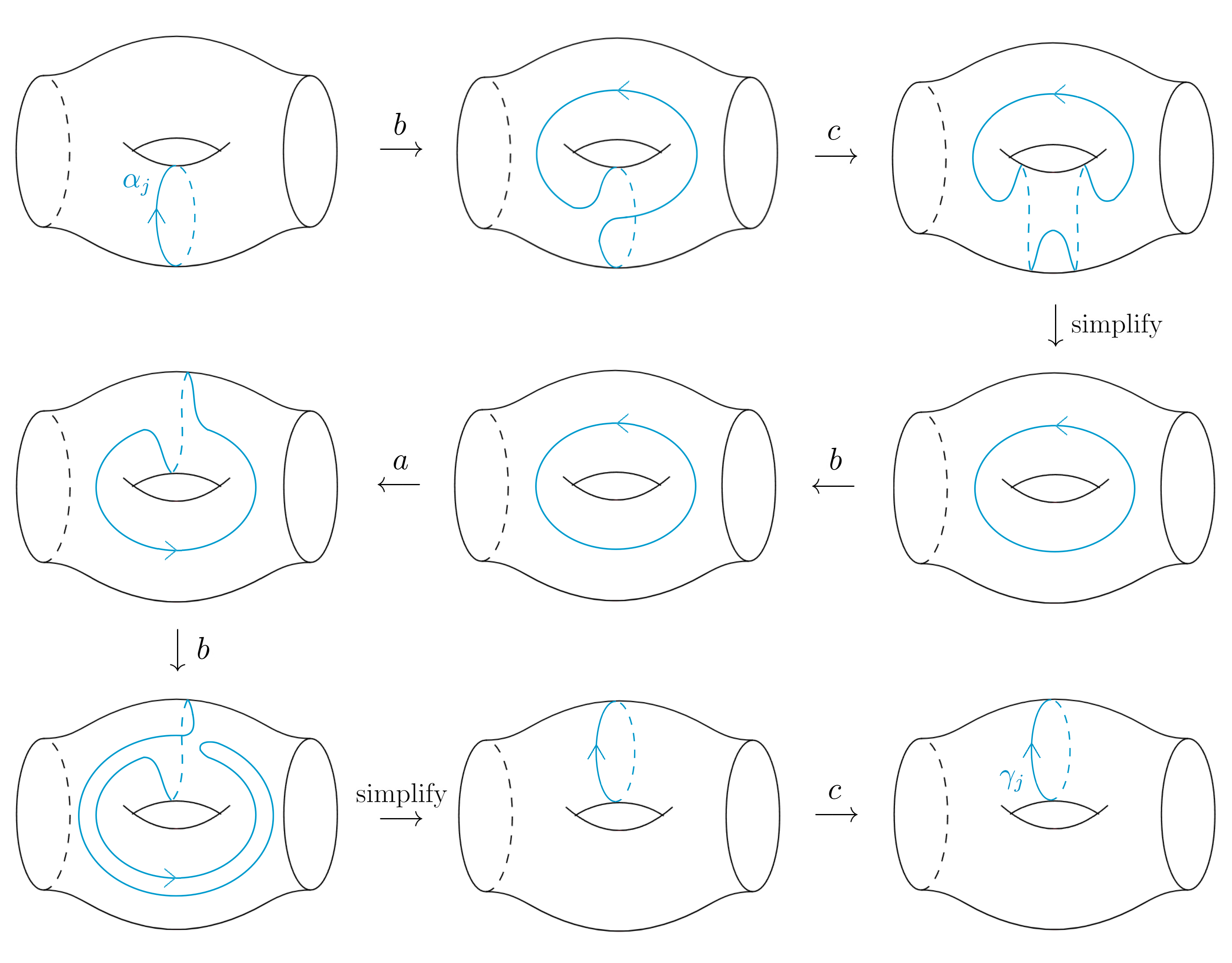}  
  \label{Figure alpha}
\end{center}

\begin{center}  
  \captionof{figure}{}
  \includegraphics[width=0.79\textwidth]{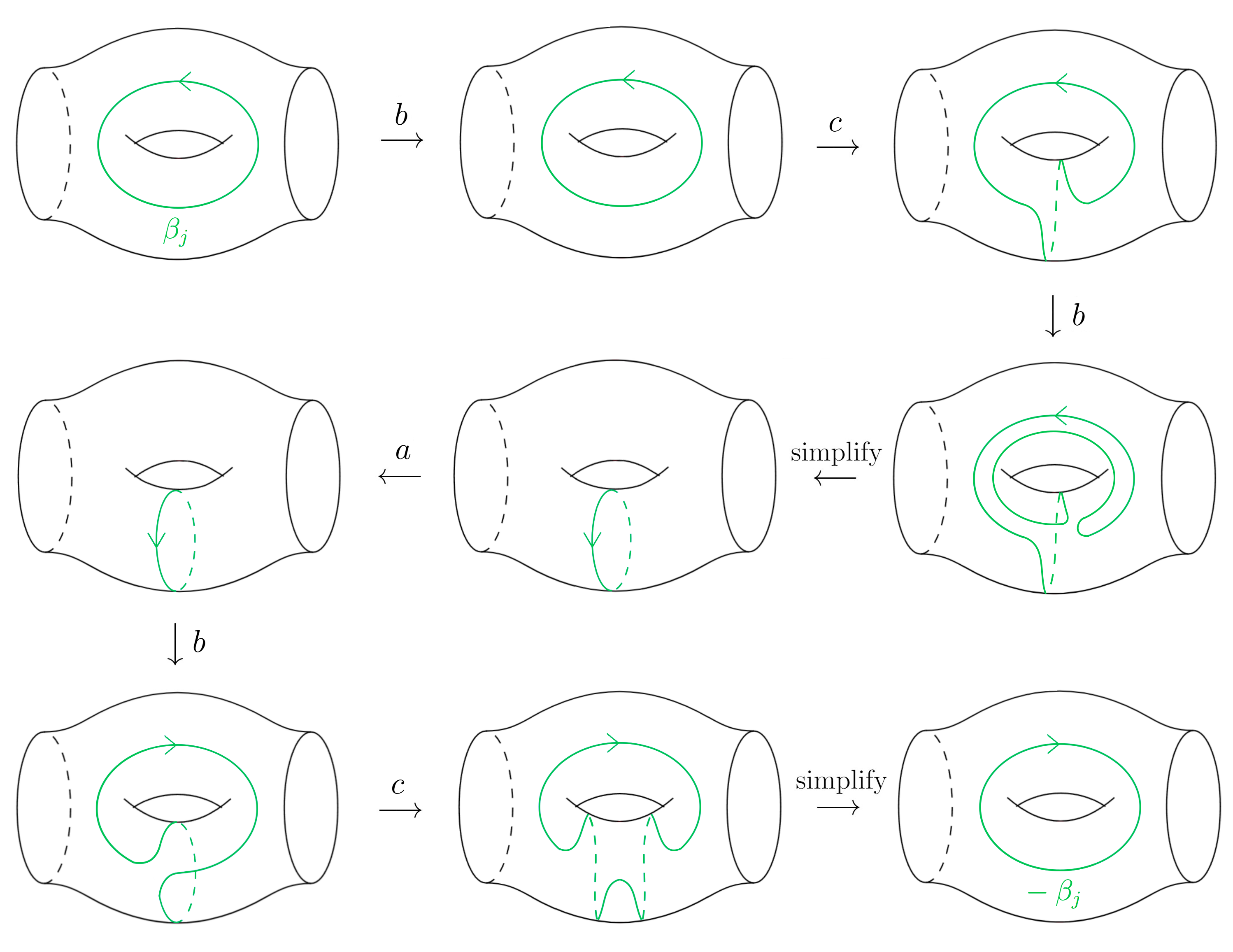}  
  \label{Figure beta}
\end{center}

Notice firstly that $\alpha_j+\gamma_j$ separates the surface $S$, meaning that $\alpha_j+\gamma_j=0\in H_1(S)$, or in other words,
\begin{align}
-\alpha_j=\gamma_j=\tau_*(\alpha_j),
\end{align}
where $\tau_*$ denotes the automorphism of $H_1(S)$ induced by $\tau$, i.e. $\tau_*=\Psi(\tau)$. Notice secondly that $\beta_j$ was sent to its inverse, or in other words,
\begin{align}
    -\beta_j=\tau_*(\beta_j).
\end{align}
We've shown that equations (4.3) and (4.4) hold for all odd $j$ inside $\{1,...,n\}$. Now, by braid and disjointness relations, we have that
\[
cbabcb=cbacbc=cbcabc=bcbabc,
\]
but then, since $cbabcb=\tau$ has order 2, we know $cbabcb$ is equal to its own inverse, and thus
\[
cbabcb=(bcbabc)^{-1}=c^{-1}b^{-1}a^{-1}b^{-1}c^{-1}b^{-1}.
\]
This shows that every right Dehn twist in $(cab)^2$ can be replaced with a left Dehn twist and vice versa, and nothing will change. Therefore, if $j\in\{1,...,n\}$ is even, $(cab)^2$ will affect the curve $\alpha_j$ and $\beta_j$ in the exact same manner as what is depicted in \blue{Figure \ref{Figure alpha}} and \blue{Figure \ref{Figure beta}}. We conclude that equations (4.3) and (4.4) hold for all $j\in\{1,...,n\}$ and 
hence, $\Psi(\tau)=\tau_*=-\mathrm{Id}_{2n}\in\mathrm{Sp}_{2n}(\mathbb{Z})$, as required.\end{proof}

The proof above can be utilized to prove a slightly more precise corollary which will come in handy in the next section. For the following result, we let $\lambda$ denote the homeomorphism of $S$ given by the following rotation by 180°:
\begin{center}  
  \captionof{figure}{}
  \includegraphics[width=0.63\textwidth]{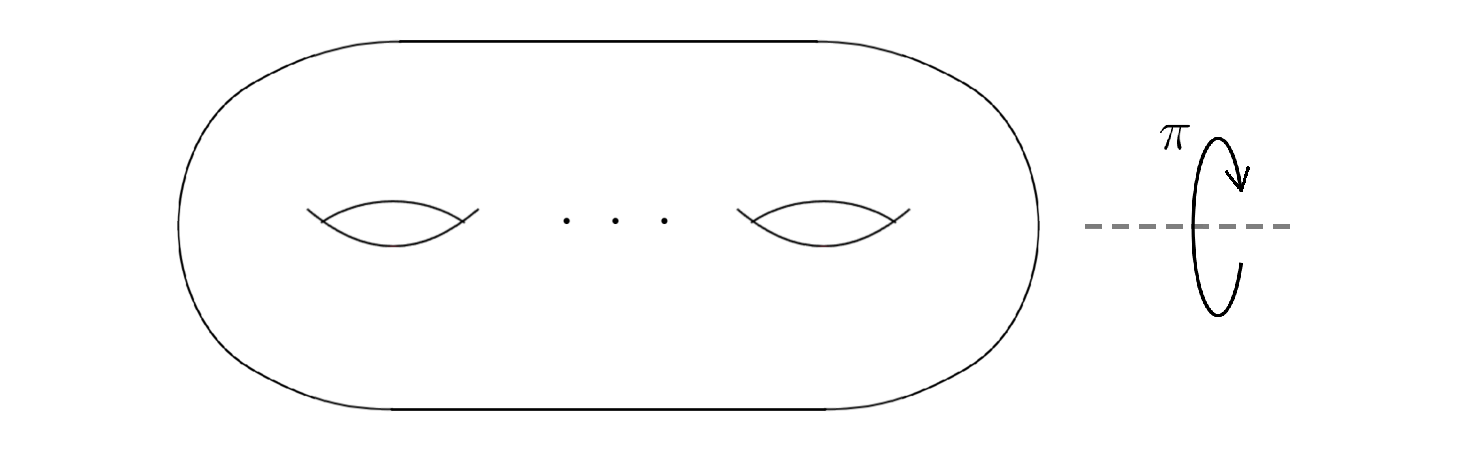}  
  \label{Figure lambda}
\end{center}
\begin{corollary}[Corollary of the previous proof]
\label{Corollary 5.1.1}
Let $S$ be a closed surface of genus $n$ and let $[\lambda]\in\mathrm{Mod}(S)$ denote the mapping class of $\lambda$. We have that $[\lambda]=(cab)^2\in\Theta_n^1$.
\end{corollary}

\begin{proof}
    Consider the following simple closed curves in $S$, which satisfy the assumptions required for the Alexander method \cite[Section 2.3]{Primer}:\samepage{
\begin{center}  
  \captionof{figure}{} \includegraphics[width=0.97\textwidth]{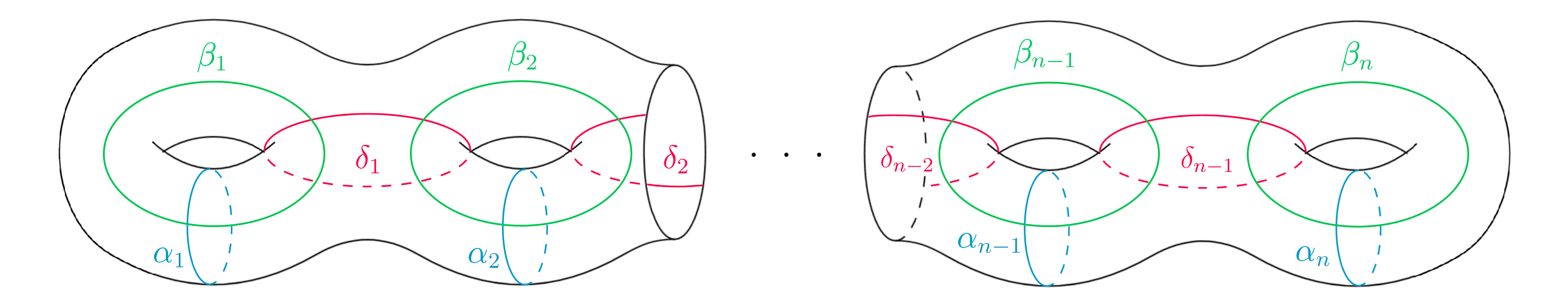}  
  \label{Figure Alexander}
\end{center}
}

\newpage

\samepage{
{\spaceskip=5pt Recall that we have $(cab)^2=cbabcb$ by braid and disjointness relations. Assume} $j\in\{1,...,n-1\}$ is odd. We shall determine how the curve $\delta_j$ is affected by the mapping class $cbabcb$ up to isotopy:

\begin{center}  
  \includegraphics[width=0.86\textwidth]{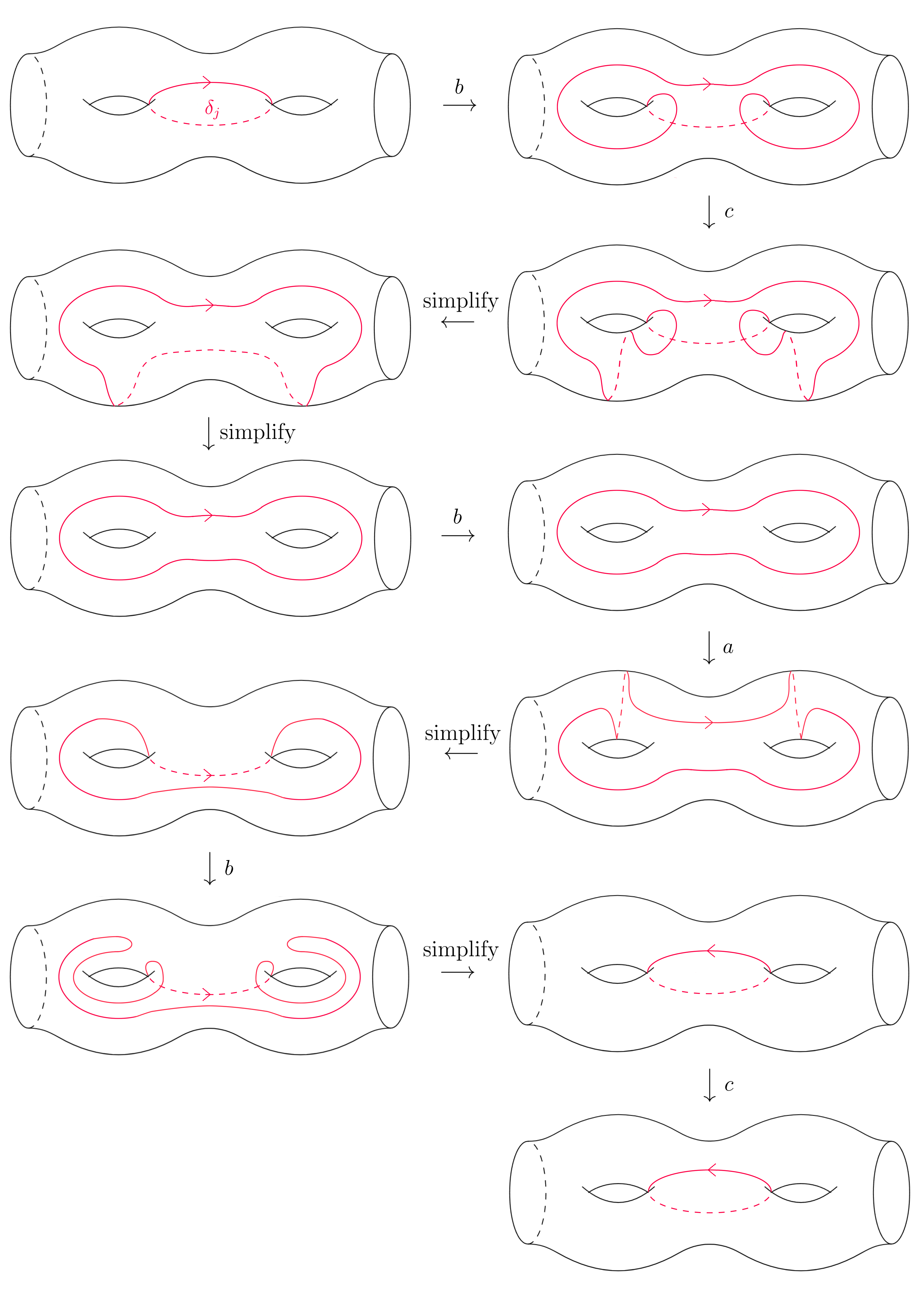}  

\end{center}

\vspace{-2pt}

By a similar argument to the one in the previous proof, the result in this figure must also hold for even $j\in\{1,...,n-1\}$, so that it holds for all $j\in\{1,...,n-1\}$. Now notice that the mapping class $[\lambda]\circ(cab)^2$ preserves $\delta_j$ along with its orientation. The same can {\spaceskip=3pt be said about $\alpha_j$ and $\beta_j$ for any $j\in\{1,...,n\}$ (recall the results in \blue{Figure \ref{Figure alpha}} and \blue{Figure \ref{Figure beta}})}. We conclude by the Alexander method that $[\lambda]\circ(cab)^2=1$, and thus
\[
[\lambda]=(cab)^2.
\]}\end{proof}

\section{Roots of hyperelliptic involutions}
\label{Section 6}
In this section, we go over the roots admitted by the expression $(abc)^2\in\mathrm{Mod}(S)$. For genus 1 and 2, these can be seen as matrices, and for genus 3 or higher, these can be seen as braids, as discussed in the previous section.

\vspace{6pt}

\subsection{Genus 1 and 2} \label{Section 5.1} Suppose $n\in\{1,2\}$. It was shown in \blue{Section \ref{Section 3}} that $\Theta_n^1$ has a faithful symplectic representation, yielding an isomorphism to $\mathrm{SL}_2(\mathbb{Z})$. As a result, the hyperelliptic involution $(aba)^2\in\Theta_n^1$ can be identified with the matrix $-\mathrm{Id}_2$, which admits many square and cubic roots inside $\mathrm{SL}_2(\mathbb{Z})$. We determine what those are in the following proposition, and record the consequences for $\Theta_n^1$ in \blue{Theorem \ref{theorem X}}.

\begin{proposition}
\label{Proposition 6.1}
The matrix $-\mathrm{Id}_2\in\mathrm{SL}_2(\mathbb{Z})$ has infinitely many square and cubic roots inside $\mathrm{SL}_2(\mathbb{Z})$. The former roots are precisely the conjugacy classes of
\begin{align*}
\begin{bmatrix}
0 & 1\\
-1 & 0
\end{bmatrix} \,\,\text{  and  }\,\,\, \begin{bmatrix}
0 & -1\\
1 & 0
\end{bmatrix},
\end{align*}
and the latter roots are precisely the conjugacy classes of
\begin{align*}
\begin{bmatrix}
0 & 1\\
-1 & 1
\end{bmatrix} \,\,\text{  and  }\,\,\, \begin{bmatrix}
1 & -1\\
1 & 0
\end{bmatrix},
\end{align*}
as well as $-\mathrm{Id}_2$ itself.
\end{proposition}
\begin{proof}
    We begin with square roots. Let $M\in\mathrm{SL}_2(\mathbb{Z})$ and write
\begin{align*}
M=\begin{bmatrix}
t_1 & t_2\\
t_3 & t_4
\end{bmatrix}
\end{align*}
for $t_1,t_2,t_3,t_4\in\mathbb{Z}$. Assume $M^2=-\mathrm{Id}_2$. Then, the eigenvalues of $M^2$ are $-1$ and $-1$, so that the eigenvalues of $M$ must square to $-1$. Also, the product of said eigenvalues must equal $\mathrm{det}(M)=1$, hence these must be $i$ and $-i$. We conclude that the characteristic polynomial of $M$ is $x^2+1$, meaning
\begin{align*}
x^2+1=\mathrm{det}(\mathrm{Id}_2\cdot x-M)=x^2-\mathrm{tr}(M)x+\mathrm{det}(M),
\end{align*}
which implies $\mathrm{tr}(M)=0$. Next, consider the action of $\mathrm{SL}_2(\mathbb{Z})$ on the upper half plane $\mathbb{H}$ given by Möbius transformations, that is
\[
Mz=\frac{t_1z+t_2}{t_3z+t_4} \,\,\,\text{   for all }z\in\mathbb{H}.
\]
Since $|\mathrm{tr}(M)|<2$, $M$ must be elliptic \cite[Section 2.2]{Topics}. This means $M$ acts as a rotation of $\mathbb{H}$ with exactly one fixed point $w\in\mathbb{H}$. We know by \cite[Theorem 1.2]{Topics} that the fundamental domain of $\mathbb{H}$ with respect to the action of $\mathrm{SL}_2(\mathbb{Z})$ is
\[
D=\{z\in \mathbb{H}\,:\,|\mathrm{Re}(z)|<\frac{1}{2},\,|z|>1\,\}.
\]
Then, by definition there exists a matrix $C\in\mathrm{SL}_2(\mathbb{Z})$ such that $Cw$ lies inside the closure of $D$. As such, the matrix $CMC^{-1}$ has fixed point $w_0\coloneq Cw$. Write
\begin{align*}
M_0\coloneq CMC^{-1}=\begin{bmatrix}
p & q\\
r & s
\end{bmatrix},
\end{align*}
for $p, q, r, s\in\mathbb{Z}$. Since $w_0$ is fixed by $M_0$, we have
\[
w_0=\frac{pw_0+q}{rw_0+s}.
\]
if $r=0$ then this forces $M_0=\pm \mathrm{Id}_2$, which contradicts the fact that $M_0$ is conjugate to an order 4 element, so $r\neq0$. We may therefore use the quadratic formula to isolate $w_0$, which specifically yields
\[
w_0=\frac{p-s}{2r} + \frac{\sqrt{-(s-p)^2-4qr}}{2r}i,
\]
since $w_0$ has positive imaginary part. As a result we obtain the following:
\begin{align}
|\mathrm{Re}(w_0)|=\frac{|p-s|}{|2r|} \,\,\,\,\text{   and   }\,\,\,\,|w_0|=\sqrt{\frac{-q}{r}}.
\end{align}
Now, since $M_0^2=-\mathrm{Id}_2$, we know from earlier that $\mathrm{tr}(M_0)=p+s=0$, which implies $p-s=2p$. Also, from the fact that $w_0$ lies in the closure of $D$, we can write
\[
\frac{|p|}{|r|}\leq \frac{1}{2}\,\,\,\,\text{ 
  and   }\,\,\,\,\sqrt{\frac{-q}{r}}\geq 1,
\]
so in particular,
\begin{align}
2|s|=2|p|\leq |r| \leq |q|.
\end{align}
Suppose for the sake of contradiction that $|r|>1$. By the reverse triangle inequality, we know
\[
||ps|-|qr||\leq|ps-qr|=|\mathrm{det}(M_0)|=1,
\]
which implies $-1\leq |ps|-|qr|$. Furthermore, equation (5.2) implies $r^2\leq|qr|$, meaning in particular that
\[
0<r^2-1\leq|qr|-1\leq|ps|=p^2.
\]
But then, 
\[
 0<r^2-1\leq p^2 \quad\implies\quad |r|\leq \sqrt{p^2+1} < 2|p|,
\]
which contradicts equation (5.2), so we conclude $|r|=1$. As a result, equation (5.2) forces $p=s=0$, and then $\mathrm{det}(M_0)=1$ forces $q=-r$. Hence, $M_0$ is one of two matrices:
\begin{align*}
\begin{bmatrix}
0 & 1\\
-1 & 0
\end{bmatrix} \,\,\text{  or  }\,\,\, \begin{bmatrix}
0 & -1\\
1 & 0
\end{bmatrix}.
\end{align*}
As a result, since $M$ was arbitrary, any matrix inside $\mathrm{SL}_2(\mathbb{Z})$ which squares to $\mathrm{-Id}_2$ must be conjugate to one of these two matrices. Conversely, if any matrix is conjugate to one of these two matrices, then it must square to $-\mathrm{Id}_2$, because
\[
(CM_0C^{-1})^2=CM_0^2C^{-1}=-C\mathrm{Id}_2C^{-1}=-\mathrm{Id}_2
\]
for all $C\in\mathrm{SL}_2(\mathbb{Z})$. Finally, notice that these two matrices are inverses of each other, and that their conjugacy classes are infinite, as for instance
\begin{align*}
\begin{bmatrix}
1 & n\\
0 & 1
\end{bmatrix}{\begin{bmatrix}
0 & 1\\
-1 & 0
\end{bmatrix}}^{\pm 1}{\begin{bmatrix}
1 & n\\
0 & 1
\end{bmatrix}}^{-1} = {\begin{bmatrix}
-n & n^2+1\\
-1 & n
\end{bmatrix}}^{\pm 1}
\end{align*}
for all $n\in\mathbb{N}$. This concludes the proof of the original claim for square roots of $-\mathrm{Id}_2$ inside $\mathrm{SL}_2(\mathbb{Z})$.
\\
\\
We now tackle the cubic roots of $\mathrm{-Id}_2$ inside $\mathrm{SL}_2(\mathbb{Z})$. Assume $M\in\mathrm{SL}_2(\mathbb{Z})$ satisfies $M^3=-\mathrm{Id}_2$. Then, the eigenvalues of $M^3$ are $-1$ and $-1$, so that the eigenvalues of $M$ must cube to $-1$. Additionally, the product of said eigenvalues must equal $\mathrm{det}(M)=1$, so we conclude the eigenvalues must be either $-1$ and $-1$ or $e^{i\pi/3}$ and $e^{-i\pi/3}$. Assume the former. Then, the characteristic polynomial of $M$ is $x^2+2x+1$, meaning
\begin{align*}
x^2+2x+1=\mathrm{det}(\mathrm{Id}_2\cdot x-M)=x^2-\mathrm{tr}(M)x+\mathrm{det}(M).
\end{align*}
 Suppose for the sake of contradiction that $M\neq -\mathrm{Id}_2$. The above equation implies $\mathrm{tr}(M)=-2$, so $M$ must be parabolic \cite[Section 2.2]{Topics}, and hence it has infinite order, a contradiction. So $M$ must be $-\mathrm{Id}_2$. 

Now assume $M$ has eigenvalues $e^{i\pi/3}$ and $e^{-i\pi/3}$ instead. Then, the characteristic polynomial of $M$ is $x^2-x+1$, meaning $\mathrm{tr}(M)=1$, so that $M$ is elliptic. By a similar argument to the one for square roots, $M$ must be conjugate to some $M_0\in\mathrm{SL}_2(\mathbb{Z})$ with fixed point $w_0$ lying inside the closure of $D$. Furthermore, letting $p,q,r,s\in\mathbb{Z}$ denote the entries in $M_0$ as before, we must have $\mathrm{tr}(M_0)=1=p+s$. The equalities in (5.1) still hold true, however this time $|p-s|=|2p-1|$. Then, since $w_0$ lies in the closure of $D$, we may write
\[
\frac{|2p-1|}{|2r|}\leq \frac{1}{2}\,\,\,\,\text{ 
  and   }\,\,\,\,\sqrt{\frac{-q}{r}}\geq 1.
\]
Notice that, by the reverse triangle inequality, we have $2|p|-1\leq|2p-1|$, and hence we can write
\begin{align}
2|p|-1\leq|r|\leq|q|.
\end{align}
Suppose for the sake of contradiction that $|p|>1$. In a similar fashion to the previous argument, we can write 
\[
r^2-1\leq|qr|-1\leq|ps|,
\]
however this time we have $|ps|=|p^2-p|\leq p^2+|p|$. As a result, we get
\[
r^2-1\leq p^2+|p|,
\]
which implies
\[
|r|\leq \sqrt{p^2+|p|+1}<2|p|-1,
\]
thereby contradicting equation (5.3). We conclude $|p|\leq 1$, and since $p+s=1$, this forces one of $p,s$ to be 1 and the other to be 0. Then, $\mathrm{det}(M_0)=1$ forces one of $q,r$ to be 1 and the other to be -1. At this stage, there are four possibilities for $M_0$:
\begin{align*}
\begin{bmatrix}
0 & -1\\
1 & 1
\end{bmatrix}, \,\,\text{  or  }\,\,\, \begin{bmatrix}
0 & 1\\
-1 & 1
\end{bmatrix},
\,\,\text{  or  }\,\,\, \begin{bmatrix}
1 & 1\\
-1 & 0
\end{bmatrix},
\,\,\text{  or  }\,\,\, \begin{bmatrix}
1 & -1\\
1 & 0
\end{bmatrix}.
\end{align*}
However, we can notice that the first matrix is the transpose of the second, which means these two are conjugate, and the same can be said of the third and fourth matrices respectively. Therefore, $M$ belongs to one of two conjugacy classes. Either that of
\begin{align*}
\begin{bmatrix}
0 & 1\\
-1 & 1
\end{bmatrix} \,\,\text{  or that of its inverse }\,\,\, \begin{bmatrix}
1 & -1\\
1 & 0
\end{bmatrix}.
\end{align*}
Then, by similar arguments to the ones used for square roots, we have that the cubic roots of $-\mathrm{Id}_2$ inside $\mathrm{SL}_2(\mathbb{Z})$ are precisely the infinite conjugacy classes of these two matrices, as well as $-\mathrm{Id}_2$ itself.
\end{proof}

\begin{theorem}
\label{theorem X}
    Let $S$ be a closed surface of genus  $n\in\{1,2\}$. Then, the hyperelliptic involution inside $\mathrm{Mod}(S)$ has infinitely many square and cubic roots inside $\Theta_n^1$. The former roots are precisely the conjugacy classes of $(aba)$ and $(aba)^{-1}$ within $\Theta_n^1$, and the latter roots are precisely the conjugacy classes of $(ab)$ and $(ab)^{-1}$ within $\Theta_n^1$, as well as the hyperelliptic involution itself.
\end{theorem}
\begin{proof}
    Let $n\in\{1,2\}$. Then, we may recall the isomorphism $\Theta_n^1\cong\mathrm{SL}_2(\mathbb{Z})$ constructed in \blue{Section \ref{Section 3}}, given by $a\longmapsto A$ and $b\longmapsto B$, where
\begin{align*}
A=\begin{bmatrix}
1 & 1\\
0 & 1
\end{bmatrix} \,\,\text{  and  }\,\,\, B=\begin{bmatrix}
1 & 0\\
-1 & 1
\end{bmatrix}.
\end{align*}
With this in mind, the unique hyperelliptic involution $(aba)^2\in\Theta_n^1$ corresponds to the matrix $(ABA)^2=-\mathrm{Id}_2\in\mathrm{SL}_2(\mathbb{Z})$. Furthermore, it is easily verified that
\begin{alignat*}{2}
ABA&=\begin{bmatrix}
0 & 1\\
-1 & 0
\end{bmatrix}, \quad (ABA)^{-1}=&&\begin{bmatrix}
0 & -1\\
1 & 0
\end{bmatrix},\\
\text{and}\quad\quad AB&=\begin{bmatrix}
0 & 1\\
-1 & 1
\end{bmatrix}, \quad
(AB)^{-1}=&&\begin{bmatrix}
1 & -1\\
1 & 0
\end{bmatrix},
\end{alignat*} which, by \blue{Proposition \ref{Proposition 6.1}}, proves the required claim.
\end{proof}

    \noindent\textit{Remark 5.2.1.} Recall that a closed surface $S$ of genus $n$ can be thought of as either a $(4n)$-gon with opposite sides identified or a $(4n+2)$-gon with opposite sides identified.
    It is easy to check that a rotation by $\pi$ of either of these polygons yields a hyperelliptic involution on $S$ (it is an order 2 mapping class which sends the generators of $H_1(S)$ to their inverses). As such, the existence of at least one square root of said hyperelliptic involution is expected, since $4n$ is always divisible by 4 and thus a rotation by $\pi/2$ of the $(4n)$-gon is possible. However, for genus $2$, the existence of cubic roots of the hyperelliptic involution (other than itself) is interesting, because the polygons involved are the octagon and the decagon, and yet neither 8 nor 10 are divisible by 6.

\vspace{6pt}

\subsection{Genus 3 and higher} \label{Section 5.2} If the surface $S$ has genus $n\geq 3$, we can produce a result analogous to \blue{Theorem \ref{theorem X}} for square roots of hyperelliptic involutions. We consider the conjugacy classes of $(abc)\in\Theta_n^1$ and its inverse.

\begin{theorem}
    {\spaceskip=3pt Let $S$ be a closed surface of genus $n\geq 3$. Then, each hyperelliptic involution} in $\mathrm{Mod}(S)$ has infinitely many square roots that are conjugate to either $(abc)\in\Theta_n^1$ or $(abc)^{-1}\in\Theta_n^1$.
\end{theorem}

\begin{proof}
Let $n\geq3$ and let $j\in\{1,...,n\}$. Recall the placement of the curve $\beta_j$ in \blue{Figure \ref{Figure Alexander}} and the definition of the homeomorphism $\lambda$ given in \blue{Figure \ref{Figure lambda}}. Notice that $\beta_j$ is fixed by $\lambda$ up to isotopy, and hence we have
\[
T_{\beta_j}=T_{\lambda(\beta_j)}=[\lambda]\circ T_{\beta_j}\circ[\lambda]
\]
by \cite[Fact 3.7]{Primer}.
We conclude that $T_{\beta_j}$ commutes with $[\lambda]$ for all $j\in\{1,...,n\}$, so by extension $b\in\Theta_n^1$ also does.

Now let $m\in\mathbb{Z}$ and let $h_m\coloneq b^m(cab)b^{-m}$. We know that $h_m$ is conjugate to $abc$ because $cab$ is. We have that
\[
h_m^2=b^m(cab)^2b^{-m}=b^m[\lambda] b^{-m}=[\lambda],
\]
where the second equality is due to \blue{Corollary \ref{Corollary 5.1.1}}.
Next, assume $\rho\in\mathrm{Mod}(S)$ is a hyperelliptic involution. Then $\rho$ is conjugate to $[\lambda]$, i.e. there exists $g\in\mathrm{Mod}(S)$ such that
\[
\rho=g[\lambda] g^{-1}=gh_m^2g^{-1}=(gh_mg^{-1})^2.
\]
Therefore, $\rho$ has at least one square root which is conjugate to $abc$. In order to show that there is an infinite number of such square roots of $\rho$, it will suffice to show that the roots of the form $gh_mg^{-1}$ are pairwise distinct, i.e. for all $m_1,m_2\in\mathbb{Z}$ we have
$h_{m_1}=h_{m_2} \implies m_1=m_2 $. If we let $m=m_1-m_2$, then we have
\[
h_{m_1}=h_{m_2}\iff h_m=h_0,
\]
So the desired statement becomes,
\begin{align*}
\forall m\in\mathbb{Z},\quad h_m=h_0\implies m=0.
\end{align*}
We will prove the contrapositive. Assume $m\neq 0$ and notice that
\[
h_m=h_0\iff b^m(ca)b^{-m}=ca,
\]
so we must show that $ca$ does not commute with $b^m$. First, it can be seen that
\[
b(\beta_2)=T_{\beta_2}(\beta_2)=\beta_2,
\]
and as a result we have
\[
cab^m(\beta_2)=ca(\beta_2).
\]
Next, for any two curves $\gamma_1,\gamma_2$ in $S$, let $i(\gamma_1,\gamma_2)$ denote the geometric intersection number of these curves. We have that the mapping class $ca$ affects the curve $\beta_2$ in the following way:

\begin{center}  
  \includegraphics[width=0.85\textwidth]{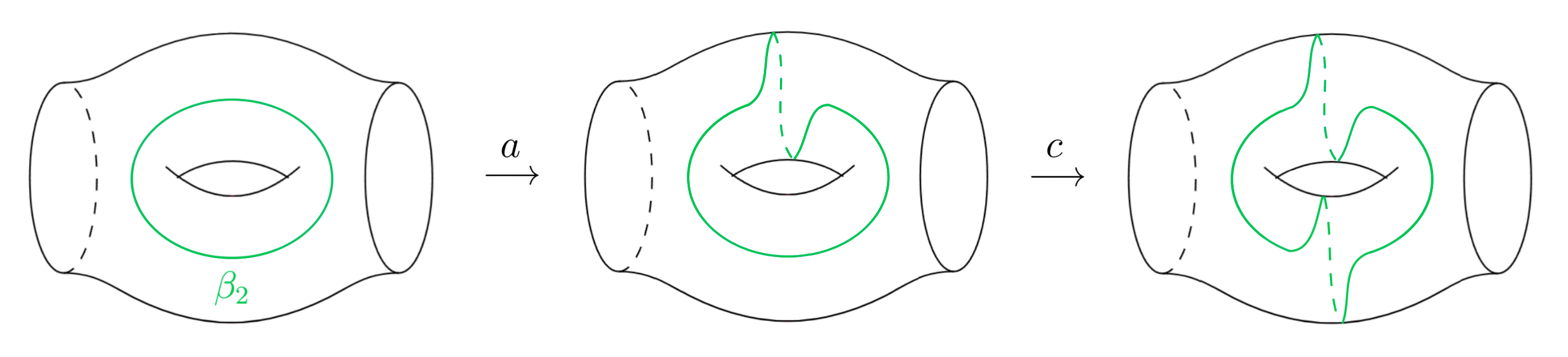}  
\end{center}

So then, by the bigon criterion \cite[Proposition 1.7]{Primer}, we have that $i(\beta_2,ca(\beta_2))\neq 0$. For notational ease, let $x=\beta_2$ and $y=ca(\beta_2)$. We have the following.
\begin{alignat*}{2}
&&i(x,y)&\neq0\\
&\implies\quad&|m|i(x,y)^2&\neq0\\
&\implies\quad& i(T_x^{\,m}(y),y)&\neq 0\\
&\implies\quad& T_x^{\,m}(y)&\neq y,
\end{alignat*}
where the second implication is due to \cite[Proposition 3.2]{Primer}. Hence we obtain
\[
b^mca(\beta_2)= T_{\beta_2}^{\,m}\circ ca(\beta_2)\neq ca(\beta_2)=cab^m(\beta_2).
\]
We have therefore shown that $ca$ and $b^m$ do not commute, as required. The original claim is proven for $abc$, but also holds for $(abc)^{-1}$ by considering $gh_m^{-1}g^{-1}$ for all $m\in \mathbb{Z}$.

\end{proof}

\vspace{2pt}

\section{Summary}
\label{Section 7}

The classification of every group of type $\Theta_n^k$ may be visualized efficiently by partitioning the $\mathbb{N}^2$ lattice. in the following figure, each point $(n,k)$ corresponds to the group $\Theta_n^k$. The group they are isomorphic to depends on the box they lie in.

\begin{center}  
  \captionof{figure}{The classification of groups of type $\Theta_n^k$.}
\includegraphics[width=0.7\textwidth]{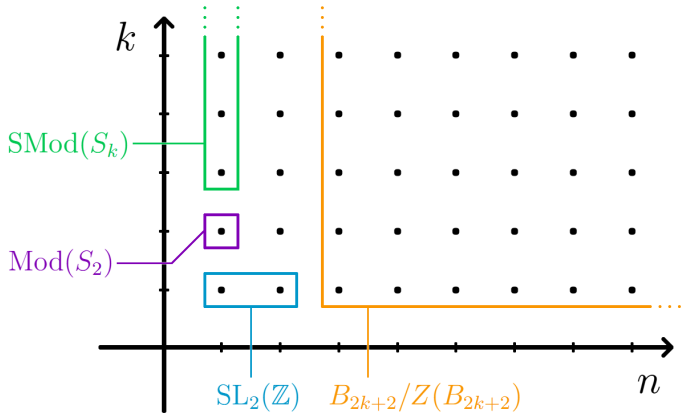}  
  \label{Figure X}
\end{center}

The blue box corresponds to the result $\Theta_1^1\cong\Theta_2^1\cong\mathrm{SL}_2(\mathbb{Z})$ proven in \blue{Section \ref{Section 3}}, and the orange box corresponds to the result $\Theta_n^k\cong B_{2k+2}/Z(B_{2k+2})$ for $n\geq 3$, proven in \blue{Section \ref{Section 4}}. Now, let $\mathrm{SMod}(S_k)$ denote the symmetric mapping classes on $S_k$ with respect to the branched cover induced by $\lambda$. Then we can notice that $\Theta_1^k=\mathrm{SMod}(S_k)$ for each $k\geq 2$ by construction, which explains the green box. That being said, the mapping classes on $S_2$ are all symmetric \cite[Section 9.4.2]{Primer}, so $\Theta_1^2=\mathrm{SMod}(S_2)=\mathrm{Mod}(S_2)$, which explains the purple box.

There is clearly one subset of the lattice which has not been classified, namely the set of groups $\Theta_2^k$ for $k\geq 2$. This remains an open question, however we can at least show that these groups don't belong to the green or orange boxes. Indeed, if $k\geq 2$, then $\Theta_2^k$ cannot be isomorphic to $\mathrm{SMod}(S_k)$, because the relation $(f_1...f_{2k-1})^{2k}=(f_{2k+1})^2$ doesn't hold in $\Theta_2^k$, whereas it does in $\mathrm{SMod}(S_k)$ (because of a chain relation). Furthermore, $\Theta_2^k$ is not isomorphic to $B_{2k+2}/Z(B_{2k+2})$ either, because the relation $(f_1...f_{2k})^{4k+2}=1$ holds in $\Theta_2^k$, but it doesn't hold in $B_{2k+2}/Z(B_{2k+2})$. Here is a quick proof. By a chain relation, we have $(f_1...f_{2k})^{4k+2}=T_{\gamma_1}^{-1}\circ T_{\gamma_2}$, where the curves $\gamma_1$ and $\gamma_2$ are given below.

\begin{center}  
  \includegraphics[width=0.48\textwidth]{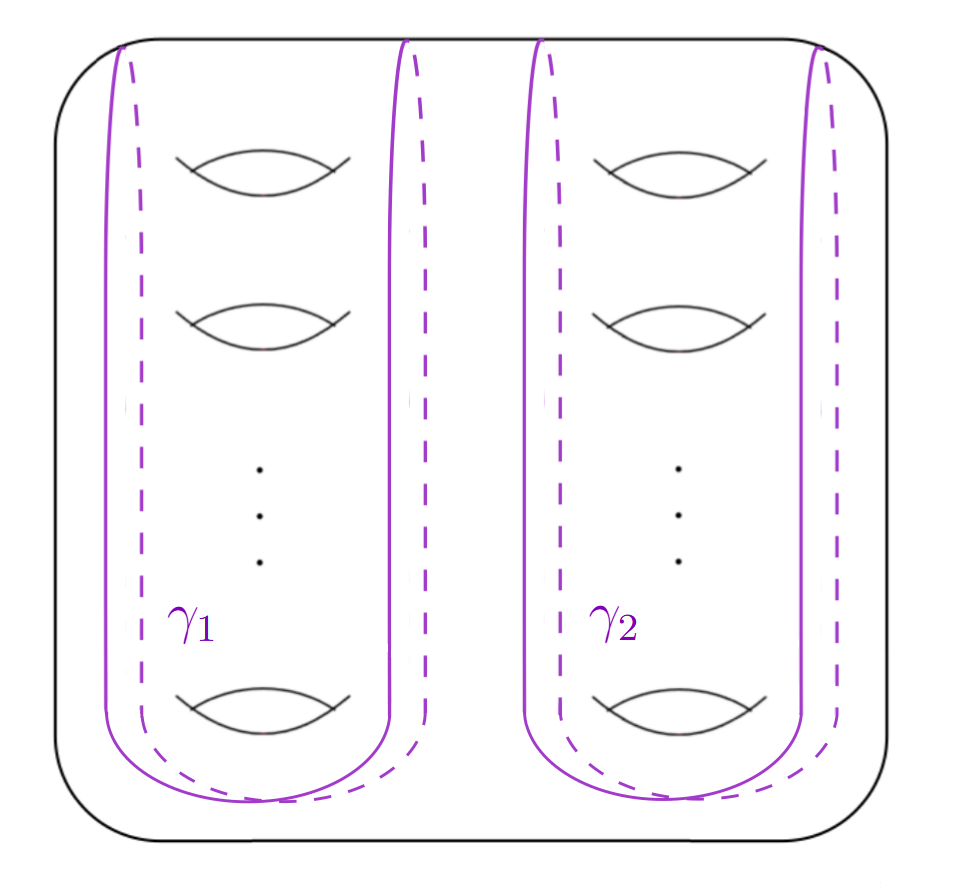}  
\end{center}

Notice that $\gamma_1$ and $\gamma_2$ are both isotopic to the curve $\boldsymbol{\varepsilon}_1$ from \blue{Figure \ref{Figure S}}, and therefore $T_{\gamma_1}^{-1}\circ T_{\gamma_2}=1$.
However, if $n\geq 3$, then by the same chain relation the word $(f_1...f_{2k})^{4k+2}$ inside $\Theta_n^k$ yields a composition of disjoint Dehn twists about essential curves $\gamma_1,...,\gamma_n$  similar to those in the above figure. These curves aren't isotopic to each other, so the Dehn twists in question don't cancel out and hence the relation $(f_1...f_{2k})^{4k+2}=1$ doesn't hold inside $B_{2k+2}/Z(B_{2k+2})$.

We can also classify some hyperelliptic involutions by using the lattice in \blue{Figure \ref{Figure X}}. Indeed, the groups of type $\Theta_n^1$ (which correspond to the bottommost row in the lattice) each contain the hyperelliptic involution $(f_1f_2f_3)^2=(abc)^2$ discussed in \blue{Section \ref{Section 5}} and the corresponding roots discussed in \blue{Section \ref{Section 6}}. Meanwhile, the groups of type $\Theta_1^k$ (which correspond to the leftmost column in the lattice) each contain the hyperelliptic involution $(f_{2k+1}...f_1f_1...f_{2k+1})$, which is the usual way of expressing a hyperelliptic involution in terms of Dehn twists (as discussed in \blue{Section \ref{Section 1}}).

\newpage

\bibliographystyle{unsrt}
\bibliography{bibliography}

\vspace{12pt}

\end{document}